\documentclass[11pt]{article}
\usepackage{latexsym,amsmath,amssymb,amsthm,amsfonts,graphicx, color}
\usepackage{graphicx}
\usepackage{epstopdf}
\numberwithin{equation}{section}
\numberwithin{figure}{section}

\usepackage[bookmarks,bookmarksnumbered]{hyperref}
\newtheorem{theorem}{Theorem}[section]
\newtheorem{lemma}[theorem]{Lemma}
\newtheorem{proposition}[theorem]{Proposition}

\numberwithin{equation}{section}
\numberwithin{equation}{section}

\Large
\begin{document}

\title{Exponential Convergence to the Maxwell Distribution Of Solutions of Spatially Inhomogeneous Boltzmann Equations}
\author{Zhou Gang\footnote{gzhou@caltech.edu, partly supported by NSF grants DMS-1308985 and DMS-1443225}}
\maketitle
\setlength{\leftmargin}{.1in}
\setlength{\rightmargin}{.1in}
\normalsize \vskip.1in \setcounter{page}{1}
\setlength{\leftmargin}{.1in} \setlength{\rightmargin}{.1in}
\centerline{Department of Mathematics, California Institute of Technology, MC 253-37, Pasadena, CA, 91106}
\section*{Abstract}
We consider the rate of convergence of solutions of spatially inhomogeneous Boltzmann equations, with hard sphere potentials, to some equilibriums, called Maxwellians.  Maxwellians are spatially homogenous static Maxwell velocity distributions with different temperatures and mean velocities. We study solutions in weighted space $L^{1}(\mathbb{R}^{3}\times \mathbb{T}^3)$. The result is that, assume the solution is sufficiently localized and sufficiently smooth, then the solution, in $L^{1}$-space, converges to a Maxwellian, exponentially fast in time. 
\tableofcontents
\section{Formulation of the problem}
In this paper we consider the Boltzmann's equation
\begin{align}\label{eq:NLBL1}
\partial_{t} g+v\cdot\nabla_{x} g= Q( g, g)
\end{align} with initial condition
$$g(v,x,0)= g_0(v,x)\geq 0, \ \ v\in \mathbb{R}^3,\  x\in \mathbb{R}^{3}/(2\pi \mathbb{Z})^3$$ satisfying
$\int_{\mathbb{R}^3\times \mathbb{T}^3} g_{0}(v,x)\ d^3 vd^3 x=1.$
The nonlinearity $Q( g, g)$ is chosen to correspond to a hard-sphere potential:
\begin{equation}\label{eq:difColli}
Q(g,g)(v,x):=\int_{\mathbb{R}^3\times \mathbb{S}^2} |(u-v)\cdot \omega|[g(u',x)g(v',x)-g(u,x)g(v,x)]\ d^3 u\ d^2 \omega,
\end{equation}
where $u', v'\in \mathbb{R}^{3}$ are given by $u':=u-[(u-v)\cdot \omega]\omega,$ $v':=v+[(u-v)\cdot\omega]\omega$.

The equation has the following properties, for any time $t\geq 0$, provided that the solution exists,
\begin{itemize}
\item[(A)] $g(v,x,t)\geq 0$ if $g_0(v,x)\geq 0;$
\item[(B)] 
\begin{itemize}
\item[(1)]
\begin{align}
\int_{\mathbb{R}^3\times \mathbb{T}^3} g(v,x,t) \ d^3 vd^3 x=\int_{\mathbb{R}^3\times \mathbb{T}^3} g_{0}(v,x)\ d^3 vd^3 x=1;\label{eq:conser1}
\end{align}
\item[(2)] 
\begin{align}
\int_{\mathbb{R}^3\times \mathbb{T}^3} v_k g(v,x,t) \ d^3 vd^3 x=\int_{\mathbb{R}^3\times \mathbb{T}^3} v_k g_{0}(v,x)\ d^3 vd^3 x,\ k=1,2,3;\label{eq:conser2}
\end{align}
\item[(3)]
\begin{align}
\int_{\mathbb{R}^3\times \mathbb{T}^3} |v|^2 g(v,x,t) \ d^3 vd^3 x=\int_{\mathbb{R}^3\times \mathbb{T}^3} |v|^2 g_{0}(v,x)\ d^3 vd^3 x.\label{eq:conser3}
\end{align}
\end{itemize}

\item[(C)] The equation has a family of Maxwellian solutions $M_{T,\mu}$ defined as
\begin{align}
M_{T,\mu}(v):=\frac{1}{(2\pi)^3}\frac{1}{(2\pi T)^{\frac{3}{2}} }e^{-\frac{|v-\mu|^2}{2T}},
\end{align} where $T$ is the temperature, and $\mu\in\mathbb{R}^3$ is the mean velocity of the gas.
\end{itemize}

The purpose of this paper is to prove {\it{asymptotic stability of Maxwellians.}} The main objective is to prove a conjecture of C. Villani, namely the solution will converge to a Maxwellian exponentially fast in a weighted $L^1-$norm, under the assumption of the smoothness and boundness (uniform in time $t$) of the solution. Specifically let $g$ be the solution to \eqref{eq:NLBL1}, then there exists a Maxwellian $M_{T,\mu}$ and some constants $m,\ C_1,\ C_2>0$ such that for any time $t\geq 0,$
\begin{align}
 \|\langle v\rangle^{m}\Big(g(\cdot,t)-M_{T,\mu}\Big)\|_{L^1}\leq C_1 e^{-C_2 t}.\label{eq:expon10}
\end{align} 
For the complete statement, see Main Theorem \ref{THM:MainTHM} below.

In the literature, one finds many results on the asymptotic stability of Maxwellians for the Boltzmann equation. One circle of results concerns the spatially homogeneous case, where $g(v,x,t)$ is independent of the position $x$. This direction of research has been pioneered by H.Grad in \cite{MR0156656}. Further results can be found in \cite{Bodmer1973,CerIllPulv,Glassey1996,Mouhot2006}. Another circle of results concerns the Boltzmann equation on an exponentially weighted $L^2$ space, namely instead of the norm in \eqref{eq:expon10}, the adopted norm is $\|M_{T,\mu}^{-\frac{1}{2}}\cdot\|_{L^2}$;  see, e.g. \cite{UKai1974,YGuo2002,YGuo2003,MR2629879, GressStra2011, MR3226836}. The advantage of working in such spaces is that spectral theory on Hilbert space can be used. There are also results in \cite{MR2209761, MR2100057, MR2679358, Briant2016, KimLee}.

In this context, the existence of weak global solutions has been established in \cite{DiPernaLion1989}. In \cite{MR2116276}, the asymptotic stability of Maxwellians, for general initial conditions, has been studied under the assumption that global smooth solutions exist. In the spatially homogeneous case, such results appear, e.g. in \cite{Ark1988,Wenn1993, Mouhot2006, MR3562318}.

There is an earlier proof of Villani's conjecture due to Maria Gualdani, Stephane Mischler and Clement Mouhot in \cite{Mouhot2010}. In the present paper an alternative proof is presented. For a non-constructive proof, see \cite{MR900501}.

In this paper, the main difficulty is to study the properties of a certain linear operator $L$ defined in Equation \eqref{eq:difL}, below. An important step in our analysis consists in proving an appropriate decay estimate for the linear evolution given by $e^{-tL}(1-P)$, where $P$ is the Riesz projection onto the eigenspace of $L$ corresponding to the eigenvalue $0$. The difficult is that, as in \cite{FrGan2012}, the spectrum of the operator $L$ occupies the entire right half of the complex plane, except for a strip of strictly positive width around the imaginary axis that only contains the eigenvalue $0$; see Figure \ref{fig:FigureExample}, below. Rewriting $e^{-tL}(1-P)$ in terms of the resolvent, $(L-z)^{-1},$ of $L$,
\begin{equation}\label{eq:spectralTHMini}
e^{-tL}(1-P)=-\frac{1}{2\pi i}\oint_{\Gamma} e^{-tz}(L-z)^{-1}\ dz,
\end{equation} (see, e.g., \cite{RSI}), where the integration contour $\Gamma$ encircles the spectrum of $L$, except for the eigenvalue $0$, we encounter the problem of proving strong convergence of the integral on the right hand side of \eqref{eq:spectralTHMini} on $L^{1}$. This problem is solved in Section \ref{sec:propagatorEst}. 

Our paper is organized as follows. 
The main Theorem will be stated in Section \ref{sec:MainTHM}. The operator obtained by linearization around Maxwellian will be derived and studied in Section \ref{sec:Formulation}. Based on the spectrum of the linear operator, the solution will be decomposed into several components. The estimates on these components will be a reformulation of the Main Theorem. This will take place in Section \ref{sec:ProofMainTHM}, and the main theorem will be proved in the same section. In the rest of the paper, namely those from Section \ref{sec:propagatorEst}, we prove the decay estimate for the propagator.

In the present paper we use the notation $a\lesssim b$ to signify that, for some fixed constant $C,$
\begin{align}
a\leq Cb.
\end{align}

\section{Main Theorem}\label{sec:MainTHM}


We start with formulating C. Villani conjecture, see \cite{RezVillani2008, MR2116276}.

The conjecture is formulated under assumptions that $g$, the solution to Boltzmann equation \eqref{eq:NLBL1}, satisfies several conditions, including the following two:
\begin{itemize}
\item[(1)] For some sufficiently large constant $\phi>0$,
\begin{align}\sup_{t\geq 0}\|\langle v\rangle^{\phi}g(\cdot, t)\|_{L^1(\mathbb{R}^3\times \mathbb{T}^3)}\lesssim 1.\label{eq:bound1}
\end{align}
\item[(2)] For some sufficiently large natural number $L$,
\begin{align}
\sup_{t\geq 0}\sum_{|k|\leq L} \|\partial_{x}^{k}g(\cdot,t)\|_{L^2(\mathbb{R}^3\times \mathbb{T}^3)}\lesssim 1.
\end{align}
\end{itemize}

By assuming these and some more assumptions, L. Desvillettes and C. Villani proved in \cite{MR2116276} that the solution converges to a Maxwellian faster than $t^{-N}$ in space $L^1$, for any $N\geq 0.$ C. Villani conjectured the convergence rate is exponential, see \cite{RezVillani2008}.

It is worth pointing out that there are examples satisfying all the assumptions, by the results of Guo in \cite{YGuo2002, YGuo2003}.

In what follows we state the main result of the present paper, which is an affirmative answer to the conjecture. We require that the initial conditions to be sufficiently close to a Maxwellian, and this is satisfied by solution at a large time, proved by C. Villani, see \cite{RezVillani2008}.

Before stating the main result, we choose $T,\ \mu$ for initial conditions $g_0$. Recall that $g_0$ is the initial conditions for Boltzmann equation \eqref{eq:NLBL1}, and $M_{T,\mu},\ T\in \mathbb{R}^{+},\ \mu\in \mathbb{R}^{3},$ are Maxwellian solutions. 
It is not difficult to see that there exist unique $T$ and $\mu$ such that
\begin{align}\label{eq:choiceTMu}
\begin{split}
\int_{\mathbb{R}^3\times \mathbb{T}^3} v_k g_0(v,x,t) \ d^3 vd^3 x=&\int_{\mathbb{R}^3\times \mathbb{T}^3} v_k M_{T,\mu} \ d^3 vd^3 x,\ k=1,2,3,\\
\int_{\mathbb{R}^3\times \mathbb{T}^3} |v|^2 g_0(v,x,t) \ d^3 vd^3 x=&\int_{\mathbb{R}^3\times \mathbb{T}^3} |v|^2 M_{T,\mu} \ d^3 vd^3 x. 
\end{split}
\end{align}
The main result is
\begin{theorem}\label{THM:MainTHM}
Assume the solution $g$ of Boltzmann equation satisfies the estimates in [1] and [2] above, and assume that the initial conditions $g(\cdot,0)$ is sufficiently close to a Maxwellian $M_{T_0,\mu_0}$ for some $T_0,\ \mu_0$, in the sense that for some $\delta=\delta(T_0)>0,$
\begin{align}
\| g(\cdot,0)- M_{T_{0}, \mu_{0}}\|_{L^{1}(\mathbb{R}^{3}\times \mathbb{T}^{3})}\leq \delta\label{eq:ini}.
\end{align}
Then for the $T,\ \mu$ chosen in \eqref{eq:choiceTMu}, there exist constants $C_0, \ C_1>0$, such that for any time $t\geq 0$
\begin{equation}\label{eq:expon}
\| g(\cdot,t)- M_{T, \mu}\|_{L^{1}(\mathbb{R}^{3}\times \mathbb{T}^{3})}\leq C_{1} e^{-C_0 t}.
\end{equation} 
\end{theorem}
This theorem will be proven in Section \ref{sec:ProofMainTHM}.

\section{The linearization around the Maxwellian}\label{sec:Formulation}
We start with linearizing around the Maxwellian solutions to obtain a linear operator.

Recall that $M_{T,\mu}$ are solutions to the equation
\begin{align}
-v\cdot\nabla_x g+Q(g,g)=0.
\end{align} We plug $g=M_{T,\mu}+f$ into the nonlinear operator $-v\cdot \nabla_x g+Q(g,g)$ to find
\begin{align}
-v\cdot\nabla_x g+Q(g,g)=-L_{T,\mu} f+ Q(f,f).
\end{align} Here the linear operator $L$ is defined by
\begin{equation}\label{eq:difL}
L_{T,\mu}:=v\cdot\nabla_{x}+ \nu_{T,\mu}(v)+K_{T,\mu}.
\end{equation} 
where $\nu_{T,\mu}$ is the multiplication operator defined by 
\begin{align}
\nu_{T,\mu}(v):=\int_{\mathbb{R}^3\times \mathbb{S}^2}|(u-v)\cdot\omega| M_{T,\mu}(u)\ d^3u d^2 \omega,\label{eq:difNu}
\end{align} and $K_{T,\mu}$ is an integral operator, defined by, for any function $f$,
\begin{align}
K_{T,\mu}(f)
:=& M_{T,\mu}(v) \int_{\mathbb{R}^{3}\times \mathbb{S}^2} |(u-v)\cdot\omega|f(u)\ d^3 ud^2 \omega\nonumber\\
& - \int_{\mathbb{R}^3\times\mathbb{S}^2 }|(u-v)\cdot\omega |M_{T,\mu}(u') f(v')\ d^3 ud^2 \omega\label{eq:difK2}\\
& -\int_{\mathbb{R}^3 \times \mathbb{S}^2} |(u-v)\cdot\omega| M_{T,\mu}(v') f(u')\ d^3 ud^2 \omega\nonumber\\
=: & K_{1}-K_2-K_3\nonumber
\end{align} where the operators $K_{l},\ l=1,2,3,$ are naturally defined.

Next we study the eigenvectors and eigenvalues of the operator $L_{T,\mu}$. By the fact that
\begin{align}
-v\cdot \nabla_x cM_{T,\mu}+Q(cM_{T,\mu},cM_{T,\mu})=0.
\end{align} for any $c\in \mathbb{R}$, $T>0$, $\mu \in \mathbb{R}^3,$ we obtain, after taking $c$, $T$ and $\mu$ derivatives on the equation above, that $L_{T,\mu}$ has five eigenvectors with eigenvalues zero
\begin{align}
M_{T,\mu},\ \partial_{T} M_{T,\mu},\ \partial_{\mu_k}M_{T,\mu},\ k=1,2,3.\label{eq:eigenvectors}
\end{align} 

A key fact is that these are the only eigenvectors for $L_{T,\mu}$ with eigenvalue $0$ in certain weighted $L^2$ space, see \cite{CerIllPulv, YGuo2002, YGuo2003, MR2301289}.

Define its Riesz projection, onto the eigenvector space, by $P^{T, \mu}$. It takes the form, for any function $h,$
\begin{align}
P^{T,\mu} h:= &\frac{1}{8\pi^{\frac{7}{2}}T^{\frac{2}{3}}} e^{-\frac{|v-\mu|^2}{2T}} \int_{\mathbb{R}^{3}\times \mathbb{T}^{3}} h(u,x)\ d^3 u d^3 x\label{eq:difProjection}\\
+&\frac{1}{\int_{\mathbb{R}^{3}}u_1^2 e^{-\frac{|u|^2}{2T}}\ d^3 u}\frac{1}{8\pi^{\frac{7}{2}}}\sum_{k=1}^{3} e^{-\frac{|v-\mu|^2}{2T}} (v_k-\mu_k)\int_{\mathbb{R}^{3}\times \mathbb{T}^{3}}(u_k-\mu_k) h(u,x)\ d^3 u d^3 x\nonumber\\
+&\frac{1}{\int_{\mathbb{R}^{3}}(|u|^2-3T)^2 e^{-\frac{|u|^2}{2T}}\ d^3 u}\frac{1}{8\pi^{\frac{7}{2}}} e^{-\frac{|v-\mu|^2}{2T}} (|v-\mu|^2-3T)\int_{\mathbb{R}^{3}\times \mathbb{T}^{3}}(|u-\mu|^2-3T) h(u,x)\ d^3 u d^3 x.\nonumber
\end{align} 

To prepare for our analysis, we state some estimates on the nonlinearity $Q$ and the operators $\nu_{T,\mu},$ $K_{T,\mu}$. 
Define a constant $\Lambda_{T}$ as
\begin{align}
\Lambda_{T}:=\inf_{v} \nu_{T,\mu}(v).\label{eq:LambdaT}
\end{align}
The results are:
\begin{lemma}\label{LM:EstNonline}
The constant $\Lambda_{T}$ is positive, i.e.
\begin{align}
\Lambda_{T}>0.\label{eq:lowB}
\end{align}
There exists a positive constant $C_{T}$ such that $\nu_{T,\mu}$ has a lower bound,
\begin{equation}\label{eq:globalLower}
\nu_{T,\mu}(v)\geq C_{T}(1+|v-\mu|).
\end{equation}
For any $m\geq 0,$ there exists a positive constant $\Upsilon_{m,T}$ such that, for any functions $f,\  g
\in L^{1}(\mathbb{R}^{3}),$
\begin{equation}\label{eq:mK1}
\sum_{l=1}^{3}\|\langle v-\mu\rangle^{m}K_{l} f\|_{L^{1}(\mathbb{R}^{3})}\leq \Upsilon_{m, T} \|\langle v-\mu\rangle^{m+1}f\|_{L^{1}(\mathbb{R}^{3})},
\end{equation}
and
\begin{align}
\|\langle v\rangle^{m} Q(f,g)\|_{L^{1}(\mathbb{R}^{3})}
\leq C_{m} \Big[\|f\|_{L^{1}(\mathbb{R}^{3})}\|\langle v\rangle^{m+1}g\|_{L^{1}(\mathbb{R}^{3})}+\|\langle v\rangle^{m+1}f\|_{L^{1}(\mathbb{R}^{3})}
\|g\|_{L^{1}(\mathbb{R}^{3})}\Big].\label{eq:estNonL}
\end{align}
\end{lemma}
This lemma is proven in Appendix \ref{sec:nonlinear}.

\section{Proof of Main Theorem \ref{THM:MainTHM}}\label{sec:ProofMainTHM}

To facilitate later analysis we reformulate equation \eqref{eq:NLBL1} into a more convenient form. 

For the $T,\ \mu$ chosen in \eqref{eq:choiceTMu}, we define a function $f:\ \mathbb{R}^{3}\times\mathbb{T}^{3}\times \mathbb{R}^{+}\rightarrow \mathbb{R}$ by
\begin{equation}\label{eq:difF}
f(v,x,t):= g(v,x,t)-M_{T,\ \mu}(v).
\end{equation} 
By the conservation laws in \eqref{eq:conser1}-\eqref{eq:conser3}, we have, since $f=g-M_{T,\mu}$,
\begin{align}
\int_{\mathbb{R}^3\times \mathbb{T}^3} v_k f(v,x,t) \ dv^3 dx^3=0,\ k=1,2,3,  
\end{align}
\begin{align}
\int_{\mathbb{R}^3\times \mathbb{T}^3} |v|^2f(v,x,t) \ dv^3 dx^3=0,
\end{align}
and by the fact $\int_{\mathbb{R}^3\times \mathbb{T}^3} g(v,x,t) \ dv^3dx^3=\int_{\mathbb{R}^3\times \mathbb{T}^3} M_{T,\mu}(v) \ dv^3dx^3=1,$
\begin{align}
\int_{\mathbb{R}^3\times \mathbb{T}^3} f(v,x,t) \ dv^3dx^3=0.
\end{align} 

These orthogonality conditions on $f$ and the definition of $P^{T, \mu}$ in \eqref{eq:difProjection} imply
\begin{align}
P^{T, \mu} f(\cdot,t)=0.\label{eq:ortho11}
\end{align}

In what follows we derive effective governing equation for $f$.
Plug the decomposition of $g$ in \eqref{eq:difF} into Boltzmann equation \eqref{eq:NLBL1} to derive
\begin{equation}\label{eq:ALinear}
 \partial_{t}f=-L_{T,\mu}f+ Q(f,f).
\end{equation}
Here the linear operator $L_{T,\mu}$ is defined in \eqref{eq:difL}, and the nonlinear term $Q(f,f)$ is defined in \eqref{eq:difColli}.

To cast the equation for $\partial_{t}f$ into a convenient form, we apply the operator $1-P^{T,\mu}$ on both sides of \eqref{eq:ALinear}, and use that $P^{T,\mu}f=0$, and that $P^{T,\mu}$ commutes with $L_{T,\mu},$ to obtain an effective equation for $f,$
\begin{align}
\partial_{t}f=-L_{T,\mu}f+(1-P^{T,\mu}) Q(f,f).\label{eq:effective10}
\end{align}
Apply Duhamel's principle on \eqref{eq:effective10} to obtain
\begin{align}
f=e^{-tL_{T,  \mu}}f_0+\int_{0}^{t}e^{-(t-s)L_{T,\mu}}(1-P^{T,\mu})Q(f,f)(s)\ ds.\label{eq:duha}
\end{align}

The proof that $f$ decays exponentially fast in weighted $L^1$ norm, relies critically on the decay estimates of the propagator $e^{-tL_{T,\mu}}(1-P^{T,\mu})$ acting on $L^{1}$. The result is
\begin{theorem}\label{THM:propagator}
If $m>0$ is sufficiently large, then there exist constants $C_0,\ C_1,\ \Pi>0$, such that for any function $h$, we have
\begin{align}
\|\langle v-\mu\rangle^{m}e^{-tL_{T,\mu}}(1-P^{T,\mu})h\|_{L^{1}(\mathbb{R}^{3}\times \mathbb{T}^{3})}\leq C_{1}e^{-C_0 t} \|\langle v-\mu\rangle^{m+\Pi}h\|_{L^{1}(\mathbb{R}^{3}\times \mathbb{T}^{3})}.\label{eq:propa}
\end{align}
\end{theorem}

The theorem will be proved in Section \ref{sec:propagatorEst}.

We continue to study the equation \eqref{eq:duha}.
Apply the propagator estimate in Theorem \ref{THM:propagator} and use that $(1-P^{T,\mu})f_0=f_0$ to find, 
\begin{align}
\|\langle v-\mu\rangle^{m}f(\cdot,t)\|_{L^{1}}\lesssim  e^{-C_{0}t}\|\langle v-\mu\rangle^{m+\Pi}f_0\|_{L^{1}}+\int_{0}^{t} e^{-C_{0}(t-s)}
\|\langle v-\mu\rangle^{m+\Pi}Q(f,f)(s)\|_{L^{1}}\ ds. \label{eq:estPre}
\end{align}

Now we estimate the terms on the right hand side.

For the second term we will prove in Subsection \ref{sec:embedding} below, together with the assumptions on the solution in Theorem \ref{THM:MainTHM}, that
\begin{align}
\|\langle v-\mu\rangle^{m+\Pi}Q(f,f)(s)\|_{L^{1}}\lesssim
 \|\langle v-\mu\rangle^m f(s)\|^{\frac{5}{4}}_{L^{1}}\leq e^{-\frac{5}{4}C_0 s} \mathcal{M}^{\frac{5}{4}}(t)\label{eq:durNon}
\end{align}
where $\mathcal{M}$ is a controlling function defined as
\begin{align}
\mathcal{M}(t):=\displaystyle\max_{0\leq s\leq t} e^{C_{0}s}  \|\langle v-\mu\rangle^{m}f(s)\|_{L^{1}(\mathbb{R}^3\times \mathbb{T}^3)}.\label{eq:difCont}
\end{align} 
It is not hard to see that $\|\langle v-\mu\rangle^{m}f(t)\|_{L^{1}(\mathbb{R}^3\times \mathbb{T}^3)}$ is continuous in $t$ by the identity in \eqref{eq:duha} and the assumptions on the solution. Thus
the function $\mathcal{M}$ is also continuous.

Suppose \eqref{eq:durNon} holds, then by \eqref{eq:estPre} 
\begin{align}
\|\langle v-\mu\rangle^{m}f(\cdot,t)\|_{L^{1}}  \lesssim e^{-C_0 t}\Big[\|\langle v-\mu\rangle^{m+\Pi}f_0\|_{L^{1}}+\mathcal{M}^{\frac{5}{4}}(t)\Big].
\end{align}
Observe that $\mathcal{M}$ is an increasing function by definition, hence
\begin{align}
\mathcal{M}(t)\lesssim \|\langle v-\mu\rangle^{m+\Pi}f_0\|_{L^{1}}+\mathcal{M}^{\frac{5}{4}}(t).\label{eq:gronwall}
\end{align}

Now we are ready to prove the main Theorem \ref{THM:MainTHM}.

The choice of initial conditions and the assumption $\|\langle v-\mu\rangle^{2m+2\Pi} f_0\|_{L^1}^{\frac{1}{2}}\lesssim 1$ in Theorem \ref{THM:MainTHM} imply, $$\mathcal{M}(0)\ll 1\ \text{and}\ \|\langle v-\mu\rangle^{m+\Pi}f_0\|_{L^{1}(\mathbb{R}^3\times \mathbb{T}^3)}\leq \|f_0\|_{L^1}^{\frac{1}{2}} \|\langle v-\mu\rangle^{2m+2\Pi} f_0\|_{L^1}^{\frac{1}{2}} \ll 1,$$
where we used the condition $\|f_0\|_{L^1}^{\frac{1}{2}}\ll 1$ in \eqref{eq:ini}.
This together with \eqref{eq:gronwall}, and that $\mathcal{M}$ is a continuous function, implies that there exists a constant $C$ such that for any $t\in [0,\infty),$
\begin{align}
\mathcal{M}(t)\leq 2C\|\langle v-\mu\rangle^{m+\Pi}f_0\|_{L^{1}(\mathbb{R}^3\times \mathbb{T}^3)}. \label{eq:controlM2}
\end{align}  

This, together with the definition of $\mathcal{M}$ in \eqref{eq:difCont}, proves Theorem \ref{THM:MainTHM}.
\begin{flushright}
$\square$
\end{flushright}



\subsection{Proof of \eqref{eq:durNon}}\label{sec:embedding}

Here we only prove \eqref{eq:durNon} for $\mu=0.$ When $\mu\not=0$, the desired estimate follows by a translation.

We claim the following two estimates
\begin{align}
\|\langle v\rangle^{m+\Pi}Q(f,f)\|_{L^{1}}\lesssim
\displaystyle\sum_{|\beta|\leq 4}\|\langle v\rangle^{m+\Pi+1}\nabla_x^{\beta}f\|_{L^{1}} \|f\|_{L^{1}},\label{eq:claim2}
\end{align}
and for $|\beta|\leq 4,$
\begin{align}
\|\langle v\rangle^{m+\Pi+1}\nabla_x^{\beta}f\|_{L^{1}}\lesssim \|\langle v\rangle^{m}f\|_{L^{1}}^{\frac{1}{4}} \Big[ 1 +\|\langle v\rangle^{3m+4\Pi+24}f\|_{L^{1}}+\|(-\Delta_{x}+1)^{20}\ f\|_{L^{2}}\Big].\label{eq:claim3}
\end{align}

Suppose that \eqref{eq:claim2} and \eqref{eq:claim3} hold, then we apply the assumptions (1) and (2) in Main Theorem \ref{THM:MainTHM} to obtain, 
\begin{align}\label{eq:assum2}
\begin{split}
\|(-\Delta_{x}+1)^{20}\ f\|_{L^{2}}\leq& \|(-\Delta_{x}+1)^{20}\ g\|_{L^{2}}+\|(-\Delta_{x}+1)^{20}\ M_{T,\mu}\|_{L^{2}}\lesssim 1,\\
\|\langle v\rangle^{3m+4\Pi+24}f\|_{L^{1}}\leq& \|\langle v\rangle^{3m+4\Pi+24}g\|_{L^{1}}+\|\langle v\rangle^{3m+4\Pi+24}M_{T,\mu}\|_{L^{1}}\lesssim 1,
\end{split}
\end{align} where, recall that $g=M_{T,\mu}+f$ by \eqref{eq:difF}.

Plug \eqref{eq:assum2} into \eqref{eq:claim3} to find
\begin{align}
\sum_{|\beta|\leq 4}\|\langle v\rangle^{m+\Pi+1}\nabla_x^{\beta}f\|_{L^{1}}\lesssim \|\langle v\rangle^{m}f\|_{L^{1}}^{\frac{1}{4}}.
\end{align}
This together with \eqref{eq:claim2} implies the desired estimate for $\|\langle v\rangle^{m+\Pi}Q(f,f)\|_{L^{1}}$, or  \eqref{eq:durNon}.

To complete the proof we need to prove \eqref{eq:claim2} and \eqref{eq:claim3} .

We start with proving \eqref{eq:claim2}.
Use \eqref{eq:estNonL} to find,
\begin{align*}
\|\langle v\rangle^{m+\Pi}Q(f,f)\|_{L^{1}(\mathbb{R}^3\times \mathbb{T}^3)}\lesssim &\big\| \|\langle v\rangle^{m+\Pi+1}f\|_{L^{1}(\mathbb{R}^3)}\|f\|_{L^{1}(\mathbb{R}^3)}  \big\|_{L^1(\mathbb{T}^3)}\\
\leq& \|f\|_{L^{1}(\mathbb{R}^3\times \mathbb{T}^3)}\ \max_{x\in \mathbb{T}^3} \|\langle v\rangle^{m+\Pi+1}f(\cdot,x)\|_{L^{1}(\mathbb{R}^3)}.
\end{align*}
The key observation is that the second factor satisfies the estimate
\begin{align}
\max_{x\in \mathbb{T}^3} \|\langle v\rangle^{m+\Pi+1}f(\cdot,x)\|_{L^{1}(\mathbb{R}^3)}\lesssim \sum_{|\beta|\leq 4}\|\langle v\rangle^{m+\Pi+1}\nabla_x^{\beta}f\|_{L^{1}}.\label{eq:claimmPi1}
\end{align} This together with the estimates above implies the desired \eqref{eq:claim2}.

To see \eqref{eq:claimmPi1}, we Fourier-expand $f$ into the form
\begin{align*}
f(v,x)=\sum_{\bf{n}\in \mathbb{Z}^3} e^{i\bf{n}\cdot x}f_{\bf{n}}(v)
\end{align*} and compute directly to have
\begin{align}
\max_{x\in \mathbb{T}^3} \|\langle v\rangle^{m+\Pi+1}f(\cdot,x)\|_{L^{1}(\mathbb{R}^3)}\leq &\sum_{\bf{n}\in \mathbb{Z}^3} \|\langle v\rangle^{m+\Pi+1}f_{\bf{n}}\|_{L^{1}(\mathbb{R}^3)}\nonumber\\
= &\sum_{{\bf{n}}\in \mathbb{Z}^3}\frac{1}{(|{\bf{n}}|^2+1)^2} (1+|{\bf{n}}|^{2})^2\|\langle v\rangle^{m+\Pi+1}f_{\bf{n}}\|_{L^{1}(\mathbb{R}^3)}.\label{eq:diffN}
\end{align}
Observe that
\begin{align*}
(|{\bf{n}}|^{2}+1)^2\|\langle v\rangle^{m+\Pi+1}f_{\bf{n}}\|_{L^{1}(\mathbb{R}^3)}=&\frac{1}{(2\pi)^3}\Big\|\Big\langle e^{i\bf{n}\cdot x}, (-\Delta_{x}+1)^2 \langle v\rangle^{m+\Pi+1}f\Big\rangle_{\mathbb{T}^3}\Big\|_{L^{1}(\mathbb{R}^3)}\\
\leq & \frac{1}{(2\pi)^3}\|(-\Delta_{x}
+1)^2 \langle v\rangle^{m+\Pi+1}f\|_{L^{1}(\mathbb{R}^3\times \mathbb{T}^3)}.
\end{align*}
Put this back into \eqref{eq:diffN} and use the fact that $\displaystyle\sum_{{\bf{n}}\in \mathbb{Z}^3}\frac{1}{(|{\bf{n}}|^2+1)^2}<\infty$ to obtain the desired \eqref{eq:claimmPi1}.

Next we prove \eqref{eq:claim3}, which is to control $\|\langle v\rangle^{m}\nabla_x^{\beta}f\|_{L^{1}},\ |\beta|\leq 4.$
The key step is to prove, for any constant $\epsilon>0,$
\begin{align}
\|\langle v\rangle^{m+\Pi+1}\nabla_x^{\beta}f\|_{L^{1}} \lesssim &\Big\|\Big[\frac{1}{\epsilon}\langle v\rangle^{2m+2\Pi+12}+\epsilon\langle v\rangle^{-10}(-\Delta_{x}+1)^{20}\Big]f \Big\|_{L^{1}}.\label{eq:holder}
\end{align}
Instead of proving this directly, we find an equivalent form, 
\begin{align}
\epsilon \Big\|\langle v\rangle^{m+\Pi+1}\nabla_x^{\beta}\Big[\langle v\rangle^{2m+2\Pi+22}+\epsilon^2(-\Delta_{x}+1)^{20}\Big]^{-1} g\Big\|_{L^{1}}\lesssim \|g\|_{L^1},
\end{align}
where $g$ is defined as $$g:=\Big[\frac{1}{\epsilon}\langle v\rangle^{2m+2\Pi+12}+\epsilon\langle v\rangle^{-10}(-\Delta_{x}+1)^{20}\Big]f.$$
The latter is easier to prove. We
Fourier-expand $g$ to be $$g(v,x)=\sum_{\bf{n}\in \mathbb{Z}^3}e^{i\bf{n}\cdot x}g_{{\bf{n}}}(v).$$ 
Then compute directly to obtain the desired result, recall that $|\beta|\leq 4,$
\begin{align}
&\epsilon \Big\|\langle v\rangle^{m+\Pi+1}\nabla_x^{\beta}\Big[\langle v\rangle^{2m+2\Pi+22}+\epsilon^2(-\Delta_{x}+1)^{20}\Big]^{-1} g\Big\|_{L^{1}(\mathbb{R}^3\times \mathbb{T}^3)}\nonumber\\
\leq &\sum_{{\bf{n}}}\epsilon (1+|{\bf{n}}|)^4 \Big\|\frac{\langle v\rangle^{m+\Pi+1}}{\langle v\rangle^{2m+2\Pi+22}+\epsilon^2(1+|{\bf{n}}|^2)^{20}} g_{{\bf{n}}}\Big\|_{L^1(\mathbb{R}^3)}\nonumber\\
\leq &\sum_{{\bf{n}}}\|\frac{1}{(1+|{\bf{n}}|^2)^{6}} g_{{\bf{n}}}\|_{L^1(\mathbb{R}^3)}\nonumber\\
\lesssim &\|g\|_{L^1(\mathbb{R}^3\times \mathbb{T}^3)}
\end{align}  where in the second step we used the H\"older's inequality, and in the last step we used that $$\|g_{n}\|_{L^1(\mathbb{R}^3)}=\frac{1}{(2\pi)^3}\| \langle e^{i{\bf{n}}\cdot x},\ g\rangle_{\mathbb{T}^3}\|_{L^1(\mathbb{R}^3)}\leq \frac{1}{(2\pi)^3}\|  g\|_{L^1(\mathbb{R}^3\times \mathbb{T}^3)}$$
and 
$
\displaystyle\sum_{{\bf{n}}\in \mathbb{Z}^3} \frac{1}{(1+|{\bf{n}}|^2)^{6}}\lesssim 1.
$

After proving \eqref{eq:holder}, we
apply H\"older's inequality on both terms to obtain
\begin{align}
\|\langle v\rangle^{m+\Pi+1}\nabla_x^{\beta}f\|_{L^{1}}\lesssim& \frac{1}{\epsilon} \|\langle v\rangle^{2m+2\Pi+12}f\|_{L^{1}}+\epsilon\|\langle v\rangle^{-10}(-\Delta_{x}+1)^{20}\ f\|_{L^{1}}\nonumber\\
\lesssim& \frac{1}{\epsilon}\Big[ \frac{1}{\epsilon^2}\|\langle v\rangle^{m}f\|_{L^{1}} +\epsilon^2\|\langle v\rangle^{3m+4\Pi+24}f\|_{L^{1}}\Big]+\epsilon\|(-\Delta_{x}+1)^{20}\ f\|_{L^{2}},\label{eqn:epsilonFix}
\end{align}  where, in controlling the second term we used the facts that $\langle v\rangle^{-10}\in L^2(\mathbb{R}^3)$, and $L^2(\mathbb{T}^3)\subset L^1(\mathbb{T}^3)$. 
What is left is to set 
\begin{align}
\epsilon=\|\langle v\rangle^{m}f\|_{L^{1}}^{\frac{1}{4}}
\end{align} in \eqref{eqn:epsilonFix}, and obtain the desired result \eqref{eq:claim3}.

\section{Propagator Estimates: Proof of Theorem \ref{THM:propagator}}\label{sec:propagatorEst}

To simplify the notations, we fix the constant $T$ and vector $\mu$ to be
\begin{align}
T=\frac{1}{2},\ \mu=0
\end{align} and for the operators $L_{\frac{1}{2},0},\ \nu_{\frac{1}{2},0},\ K_{\frac{1}{2},0}$ and $P^{\frac{1}{2},0}$ and for the Maxwellian $M_{\frac{1}{2},0}$ we adopt new notations,
\begin{align}
L:=L_{\frac{1}{2},0},\ \nu:=\nu_{\frac{1}{2},0},\ K:=K_{\frac{1}{2},0},\ P:=P^{\frac{1}{2},0},\ M:=M_{\frac{1}{2},0}.\label{eq:abbre}
\end{align}

It is easy to see that our arguments, in what follows, can be easily adapted to general cases.

The proof are based on previous results in \cite{CerIllPulv, YGuo2003, YGuo2002, MR2301289}, where it was proved that the operator $L$, mapping the space $$M^{\frac{1}{2}}L^2:=\Big\{f:\mathbb{R}^3\times \mathbb{T}^3\rightarrow \mathbb{C}\ \Big|\ \|M^{-\frac{1}{2}}f\|_{L^2}<\infty\Big\}$$ into itself, has an eigenvalue $0$ with eigenvectors listed in \eqref{eq:eigenvectors}, and it has a gap with the other parts of the spectrum.
By these we establish the crucial identity \eqref{eq:IntContour} below. 

Besides these, in proving Theorem \ref{THM:propagator}, we adopt the same strategy as in \cite{FrGan2012}, to circumvent the difficulty that the spectrum of $L$ is ``too big".

We start with outlining the general strategy of the proof.

There are two typical approaches to proving decay estimates for propagators. The first one is to use the identity $$e^{-tL}(1-P)=\frac{1}{2\pi i}\oint_{\Gamma} e^{-t\lambda}(\lambda -L)^{-1}\ d\lambda$$ where the contour $\Gamma$ is a curve encircling the spectrum of $L(1-P).$ The obstacle is that the spectrum of $L(1-P)$ occupies the entire right half of the complex plane, except for a strip in a neighborhood of the imaginary axis, as illustrated in Figure \ref{fig:FigureExample} below. This makes it difficult to prove strong convergence on $L^1$ of the integral on the right hand side.
\begin{figure}[htb!]
\centering%
\includegraphics[width=8cm]{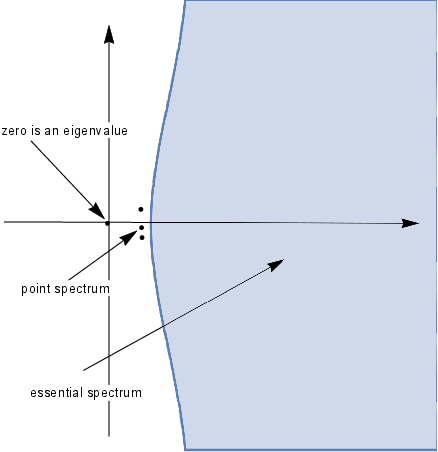}
\caption{The Spectrum of $L$}
\label{fig:FigureExample}
\end{figure}

The second approach is to use perturbation theory, which amounts to expanding $e^{-tL}$ in powers of the operator $K,$ (see \eqref{eq:difK}): $$e^{-tL}=e^{-t(\nu+v\cdot\nabla_{x})}+\int_{0}^{t}e^{-(t-s)(\nu+v\cdot\nabla_x)}K e^{-s(\nu+v\cdot\nabla_x)}\ ds+\cdots.$$ It will be shown in Proposition \ref{prop:easyEst} that each term in this expansion can be estimated quite well, but the fact that $K$ is unbounded forces us to estimate them in different spaces.

We will combine these two approaches to prove Theorem \ref{THM:propagator}, by following the steps in \cite{FrGan2012}.

We expand the propagator $e^{-tL}(1-P)$ using Duhamel's principle:
\begin{equation}\label{eq:durha}
e^{-tL}(1-P)= \sum_{k=0}^{12} (1-P)A_{k}(t)+(1-P)\tilde{A}(t),
\end{equation}
where the operators $A_{k}$ are defined recursively, with
\begin{equation}\label{eq:difA0}
A_{0}=A_{0}(t):= e^{-t (\nu+v\cdot \nabla_{x})},
\end{equation} and $A_{k},\ k=1,2,\cdots,12,$ given by
\begin{equation}
A_{k}(t):= \int_{0}^{t} e^{-(t-s)(\nu+v\cdot \nabla_{x})}K A_{k-1}(s)\ ds .
\end{equation}
Finally $\tilde{A}$ is defined by
\begin{equation}
\tilde{A}(t)=\int_{0}^{t} e^{-(t-s)L}K A_{12}(s)\ ds.
\end{equation}
The exact form of $A_{k},\ k=0,\ 1,\cdots, 12,$ implies the following estimates. 

Recall that $\Lambda:=\inf_{v}\nu(v)>0.$
\begin{proposition}\label{prop:easyEst}
For any $C_0\in (0,\ \Lambda)$, there exists a positive constant $C_1$ such that, for any function $f: \ \mathbb{R}^{3}\times \mathbb{T}^{3}\rightarrow \mathbb{C},$
\begin{equation}\label{eq:explEst}
\|\langle v\rangle^{m}A_{k}(t)f\|_{L^{1}(\mathbb{R}^{3}\times \mathbb{T}^3)}\leq C_1 e^{-C_0 t} \|\langle v\rangle^{m+k}f\|_{L^{1}(\mathbb{R}^{3}\times \mathbb{T}^3)}.
\end{equation}
\end{proposition}
This proposition is proven in Subsection \ref{subsec:ProofEEst}.

Next we estimate $\tilde{A}$, which is given by
$$\tilde{A}=\int_{0}^{t} e^{-(t-s_1)L}K \int_{0}^{s_1} e^{-(s_1-s_2) (\nu+v\cdot\nabla_{x})}K\cdots \int_{0}^{s_{12}}e^{-(s_{12}-s_{13})(\nu+v\cdot\nabla_{x})}K e^{-s_{13}(\nu+v\cdot\nabla_{x})}\ ds_{13}\cdots ds_1. $$

We start with transforming $\tilde{A}$ into a more convenient form.

One of the important properties of the operators $L$ is that, for any function $g:\ \mathbb{R}^{3}\rightarrow \mathbb{C}$ (i.e., independent of $x$) and ${\bf{n}}\in \mathbb{Z}^3,$ we have that
\begin{align}
Pe^{i{\bf{n}}\cdot x}g=&0\ \ \text{if}\ {\bf{n}}\not=0,\nonumber\\
L e^{i{\bf{n}}\cdot x}g=&e^{i{\bf{n}}\cdot x}L_{{\bf{n}}}g, \label{eq:observation}\\
(\nu+v\cdot\nabla_x)e^{i{\bf{n}}\cdot x}g=&e^{i{\bf{n}}\cdot x}(\nu+i{\bf{n}}\cdot v)g,\nonumber
\end{align} where the operator $L_{{\bf{n}}}$ is unbounded and defined as $$L_{{\bf{n}}}:=\nu+i{\bf{n}}\cdot v+K.$$ Recall that $P$ has been defined in \eqref{eq:difProjection}.

To make \eqref{eq:observation} applicable, we Fourier-expand the function $g:\ \mathbb{R}^{3}\times \mathbb{T}^{3}\rightarrow \mathbb{C}$ in the variable $x,$ i.e.,
\begin{equation}\label{eq:dirSum}
g(v,x)=\sum_{{\bf{n}}\in \mathbb{Z}^{3}} e^{i{\bf{n}}\cdot x} g_{{\bf{n}}}(v).
\end{equation} Then use \eqref{eq:observation} and compute directly to obtain
\begin{equation}\label{eq:decomA}
\|(1-P)\tilde{A}g\|_{L^{1}(\mathbb{R}^3\times \mathbb{T}^3)}\leq \sum_{{\bf{n}}\in \mathbb{Z}^{3}}\|\tilde{A}_{{\bf{n}}} g_{n}\|_{L^{1}(\mathbb{R}^3)},
\end{equation} where $\tilde{A}_{{\bf{n}}}$ is defined as follows: If ${\bf{n}}\not=(0,0,0)$ then
$$
\tilde{A}_{{\bf{n}}}:=\int_{0}^{t} e^{-(t-s_1)L_{{\bf{n}}}}K \int_{0}^{s_1} e^{-(s_1-s_2) (\nu+i v\cdot{\bf{n}})}K\cdots \int_{0}^{s_{12}}e^{-(s_{12}-s_{13})(\nu+iv\cdot{\bf{n}})}K e^{-s_{13}(\nu+iv\cdot{\bf{n}})}\ ds_{13}\cdots ds_1
$$ and for ${\bf{n}}=(0,0,0)$ we define $$
\tilde{A}_{0}:=\int_{0}^{t} (1-P)e^{-(t-s_1)L_{0}}K \int_{0}^{s_1} e^{-(s_1-s_2) \nu}K\cdots \int_{0}^{s_{12}}e^{-(s_{12}-s_{13})\nu}K e^{-s_{13}\nu}\ ds_{13}\cdots ds_1.
$$

Next, we study $\tilde{A}_{{\bf{n}}}$, which is defined in terms of the operators $e^{-tL_{{\bf{n}}}}$, $e^{-t[\nu+i {\bf{n}}\cdot v]}$ and $Ke^{-t[\nu+i{\bf{n}}\cdot v]}K.$

It is easy to estimate $e^{-t[\nu+i{\bf{n}}\cdot v]}:$ The fact that the function $\nu$ has a positive global minimum $\Lambda$ (see \eqref{eq:LambdaT}) implies that
\begin{equation}\label{eq:exactForm}
 \|e^{-t[\nu+i{\bf{n}}\cdot v]}\|_{L^{1}\rightarrow L^1}\leq e^{-\Lambda t}.
\end{equation}

Next we consider operator $e^{-tL_{{\bf{n}}}}$.

The result is:
\begin{lemma}\label{LM:roughEst}
Suppose that $m$ is sufficiently large.
There exist constants $C_0,\ C_1>0,$ such that for any time $t\geq 0$ and ${\bf{n}}\not= (0,0,0)$, we have
\begin{equation}\label{eq:est000}
\|e^{-tL_{{\bf{n}}}}\|_{\langle v\rangle^{-m}L^{1}(\mathbb{R}^{3})\rightarrow \langle v\rangle^{-m}L^{1}(\mathbb{R}^{3})} \leq C_{1}(1+|{\bf{n}}|) e^{-C_{0}t},
\end{equation} and for ${\bf{n}}=(0,0,0)$
\begin{equation}\label{eq:estNot0}
\|e^{-tL_{{\bf{0}}}}(1-P)\|_{\langle v\rangle^{-m}L^{1}(\mathbb{R}^{3})\rightarrow \langle v\rangle^{-m}L^{1}(\mathbb{R}^{3})} \leq C_{1} e^{-C_{0}t}.
\end{equation}
\end{lemma}
This lemma will be proven in Section \ref{sec:RoughEst}.

The most important step is to estimate $$K_{t}^{({\bf{n}})}:= Ke^{-t(\nu+i{\bf{n}}\cdot v)} K.$$ It is well known that the operator $K$, defined in \eqref{eq:difK2}, has an integral kernel $K(v,u)$: for any function $f:\ \mathbb{R}^3\rightarrow \mathbb{C},$ 
\begin{align}
K(f)=K_1(f)-K_2(f)-K_3(f)
\end{align} with integral kernels taking the form
\begin{align}
K_1(f)=&\pi e^{-|v|^2}\int_{\mathbb{R}^3} |u-v| f(u)\ d^3 u\label{eq:difK},\\
K_2(f)+K_3(f)=& 2\pi \int_{\mathbb{R}^3} |u-v|^{-1} e^{ -\frac{|(u-v)\cdot v|^2}{|u-v|^2}} f(u)\ d^3 u\nonumber.
\end{align}
Here we derive the integral kernels of $K_l,\ l=1,2,3,$ from the explicit form of the operator $e^{\frac{1}{2}|v|^2}Ke^{-\frac{1}{2}|v|^2}$ in \cite{Glassey1996}, (see also \cite{MR0156656, CourHil1989}).

Then the integral kernel, $K_{t}^{({\bf{n}})}(v,u),$ of $K_{t}^{({\bf{n}})}$ is given by
$$K_{t}^{({\bf{n}})}(v,u)=\int_{\mathbb{R}^{3}}K(v,z)e^{-t[\nu(z)+i{\bf{n}}\cdot z]}K(z,u)\ dz$$ for some properly defined function $K(v,u)$. The presence of the factor $e^{-it{\bf{n}}\cdot z}$ plays a critically important role. It makes the operator $K_{t}^{({\bf{n}})}$ smaller, as $|{\bf{n}}|$ becomes larger.
Recall that $\Lambda:=\inf_{v}\nu(v)>0.$
\begin{lemma}\label{LM:oscillate}
There exists a positive constant $C_{1}$ such that, for any ${\bf{n}}\in \mathbb{Z}^{3}$ and $t\geq 0,$
\begin{equation}\label{eq:oscilate}
\|K_{t}^{({\bf{n}})} f\|_{ \langle v\rangle^{-m}L^{1}(\mathbb{R}^{3})} \leq \frac{C_1}{1+|{\bf{n}}|t} e^{-\Lambda t}\|\langle v\rangle^{3} f\|_{\langle v\rangle^{-m}L^{1}(\mathbb{R}^{3})}.
\end{equation}
\end{lemma}
This lemma will be proven in Subsection \ref{subsec:oscillate}.

The results in Proposition \ref{prop:easyEst}, Lemma \ref{LM:roughEst} and Lemma \ref{LM:oscillate} suffice to prove Theorem \ref{THM:propagator}.\\
{\bf{Proof of Theorem \ref{THM:propagator}}.} In Equation \eqref{eq:durha} we have decomposed $e^{-tL}(1-P)$ into several terms. The operators $A_{k}, \ k=0,1,2,\cdots, 12,$ are estimated in Proposition \ref{prop:easyEst}.

In what follows, we study $\tilde{A}$. By \eqref{eq:decomA} we only need to control $\tilde{A}_{{\bf{n}}},\ {\bf{n}}\in \mathbb{Z}^3$. For ${\bf{n}}=(0,0,0)$ it is easy to see that
\begin{equation}\label{eq:TildeA0}
\|\langle v\rangle^{m}\tilde{A}_{0}g_{{\bf{0}}}\|_{L^{1}(\mathbb{R}^{3})}\lesssim e^{-\Lambda t} \|\langle v\rangle^{m+12} g_{\bf{0}}\|_{L^{1}(\mathbb{R}^{3})}
\end{equation} by collecting the different estimates in \eqref{eq:exactForm} and Lemma \ref{LM:roughEst} and using the estimates on $K$ in Lemma \ref{LM:EstNonline}.

For ${\bf{n}}\not=0$, we observe that the integrands in the definitions of $\tilde{A}_{{\bf{n}}}$ are products of terms $e^{-(t-s_1)L_{{\bf{n}}}},$ $K e^{-(s_k-s_{k+1})(\nu+i{\bf{n}}\cdot v)}K$ and $e^{-(s_k-s_{k+1})(\nu+i{\bf{n}}\cdot v)}$, where $k\in \{1,2,\cdots,13\}$ (we use the convention that $s_{14}=0$).
Applying the bounds in \eqref{eq:exactForm}, Lemma \ref{LM:roughEst} and Lemma \ref{LM:oscillate}, we see that there is a constant $C_0>0$ such that
\begin{align}
\begin{split}\label{eq:12times}
&\|\langle v\rangle^{m}\tilde{A}_{{\bf{n}}}g_{\bf{n}}\|_{L^{1}(\mathbb{R}^{3})}\\
\lesssim&  e^{-\Lambda t} (1+|{\bf{n}}|)\|\langle v\rangle^{m+20} g_{\bf{n}}\|_{L^{1}(\mathbb{R}^{3})} \times \\
&\int_{0}^{t}\int_{0}^{s_1}\cdots \int_{0}^{s_{12}} [1+|{\bf{n}}|(s_{12}-s_{13})]^{-1} [1+|{\bf{n}}|(s_{10}-s_{11})]^{-1}\cdots [1+|{\bf{n}}|(s_{2}-s_{3})]^{-1}\ ds_{13} ds_{12}\cdots ds_{1}.
\end{split}
\end{align}
Here a key observation is that, even though $\int_{0}^{s}\int_{0}^{s_1} [1+|{\bf{n}}|s_2]^{-1}\ ds_2ds_1$ is not bounded as $s\rightarrow \infty$, the growth in $s$ is modest, and most importantly we can get a small factor $|\bf{n}|^{-1} \ln (1+|\bf{n}| )$ since, for $s\leq t$,
\begin{align}
\begin{split}
 \int_{0}^{s}\int_{0}^{s_1} [1+|{\bf{n}}|s_2]^{-1}\ ds_2ds_1= |{\bf{n}}|^{-1} \int_{0}^{s}\ln (1+|{\bf{n}}| s_1)\ ds_1&\lesssim |{\bf{n}}|^{-1}\ln(1+|{\bf{n}}|) (1+s)^2\\
 &\leq |{\bf{n}}|^{-1}\ln(1+|{\bf{n}}|) (1+t)^2. 
\end{split}
\end{align}
Apply this to the last factor of \eqref{eq:12times} six times, and bound $(1+t)^{12}$ by $C_{\epsilon_0} e^{\epsilon_0 t}$ for any $\epsilon_0>0$, and we find, for any positive constant $\tilde{C}_0< \Lambda$, there exists a constant $C_1>0$ such that
\begin{equation*}
\|\langle v\rangle^{m}\tilde{A}_{{\bf{n}}}g_{\bf{n}}\|_{L^{1}(\mathbb{R}^{3})}\leq C_1  e^{-\tilde{C}_0t}\frac{1}{(1+|{\bf{n}}|)^{4}} \|\langle v\rangle^{m+20} g_{\bf{n}}\|_{L^{1}(\mathbb{R}^{3})}.
\end{equation*} Plugging this and \eqref{eq:TildeA0} into \eqref{eq:decomA}, we find that
\begin{equation}\label{eq:Aprelimi}
\|\langle v\rangle^{m}(1-P)\tilde{A}g\|_{L^{1}(\mathbb{R}^{3}\times\mathbb{T}^{3})} \lesssim C_1 e^{-\tilde{C}_{0}t} \sum_{{\bf{n}}\in \mathbb{Z}^{3}} \frac{1}{(1+|{\bf{n}}|)^{4}} \|\langle v\rangle^{m+20}g_{{\bf{n}}}\|_{L^{1}(\mathbb{R}^{3})}.
\end{equation} The fact $g_{n}=\frac{1}{(2\pi)^3}\langle e^{i{\bf{n}}\cdot x},\ g\rangle_{x}$ makes
\begin{equation*}
\|\langle v\rangle^{m+20}g_{{\bf{n}}}\|_{L^{1}(\mathbb{R}^{3})}\leq (2\pi)^{3} \|\langle v\rangle^{m+20}g\|_{L^{1}(\mathbb{R}^{3}\times \mathbb{T}^{3})}.
\end{equation*} This, together with the fact that $\displaystyle\sum_{{\bf{n}}\in \mathbb{Z}^{3}} \frac{1}{(1+|{\bf{n}}|)^{4}}<\infty,$ implies that
\begin{equation}\label{eq:completeA}
\|\langle v\rangle^{m}(1-P)\tilde{A}g\|_{L^{1}(\mathbb{R}^{3}\times\mathbb{T}^{3})} \lesssim C_1e^{-\tilde{C}_0 t} \|\langle v\rangle^{m+20}g\|_{L^{1}(\mathbb{R}^{3}\times\mathbb{T}^{3})}.
\end{equation}

Obviously Equation \eqref{eq:durha}, Inequality \eqref{eq:completeA} and Proposition \ref{prop:easyEst} imply Theorem \ref{THM:propagator}.
\begin{flushright}
$\square$
\end{flushright}
\subsection{Proof of Proposition \ref{prop:easyEst}}\label{subsec:ProofEEst}
Recall the definition of the constant $\Lambda=\Lambda_{\frac{1}{2}}>0$ in \eqref{eq:LambdaT}.
The definition of $A_0$ (see \eqref{eq:difA0}) implies that
\begin{equation}\label{eq:a0t}
\|\langle v\rangle^{m}A_{0}(t)f\|_{L^{1}(\mathbb{R}^{3}\times \mathbb{T}^{3})}\leq e^{-\Lambda t}\|\langle v\rangle^{m}f\|_{L^{1}(\mathbb{R}^{3}\times \mathbb{T}^{3})}.
\end{equation}

For $A_1,$ we use the estimate for the unbounded operator $K$ given in Lemma \ref{LM:EstNonline}. Compute directly to obtain
\begin{align*}
\|\langle v\rangle^{m}A_{1}(f)\|_{L^{1}(\mathbb{R}^{3}\times \mathbb{T}^{3})}\leq &\int_{0}^{t} e^{-\Lambda(t-s)}\|\langle v\rangle^{m}K e^{-s(\nu+v\cdot\nabla_{x})x}f\|_{L^{1}(\mathbb{R}^{3}\times \mathbb{T}^{3})}\ ds\\
\lesssim& \int_{0}^{t} e^{-\Lambda(t-s)} e^{-\Lambda s}\ ds \|\langle v\rangle^{m+1} f\|_{L^{1}}\\
= &e^{-\Lambda t}t \|\langle v\rangle^{m+1} f\|_{L^{1}}.
\end{align*}
Similar arguments yield the desired estimates for $A_{k},\ k=2,3,\cdots 12$.

Thus, the proof of Proposition \ref{prop:easyEst} is complete.

\begin{flushright}
$\square$
\end{flushright}

\subsection{Proof of Inequality \eqref{eq:oscilate}}\label{subsec:oscillate}
\begin{proof}
We denote the integral kernel of the operator $K$ by $K(v,u)$ and infer its explicit form from \eqref{eq:difK}.
It is then easy to see that the integral kernel of the operator $Ke^{-t(\nu+i{\bf{n}}\cdot v)}K$ is given by
$$K_{t}^{({\bf{n}})}(v,u):=\int_{\mathbb{R}^{3}}K(v,z) e^{-t [\nu(z)+i{\bf{n}}\cdot z]} K(z,u)\ d^3 z.$$

We use the oscillatory nature of $e^{-it{\bf{n}}\cdot z}$ to derive some ``smallness estimates" when $|{\bf{n}}|$ is sufficiently large, by integrating by parts in the variable $z$.
Without loss of generality we assume that $$|n_{1}|\geq \frac{1}{3}|{\bf{n}}|.$$ Integrate by parts in the variable $z_1$ to obtain
\begin{align}
K_{t}^{({\bf{n}})}(v,u)=&\int_{\mathbb{R}^{3}} K(v,z)K(z,u) \frac{1}{-t[\partial_{z_1} \nu(z)+in_1]} \partial_{z_1} e^{-t [\nu(z)+i{\bf{n}}\cdot z]} \ d^3 z\nonumber\\
=& \int_{\mathbb{R}^{3}} \partial_{z_1} [K(v,z)K(z,u) \frac{1}{t[\partial_{z_1} \nu(z)+in_1]} ] e^{-t [\nu(z)+i{\bf{n}}\cdot z]} \ d^3 z
\end{align}

The different terms in $\partial_{z_1} [K(v,z)K(z,u) \frac{1}{t[\partial_{z_1} \nu(z)+in_1]} ]$ are dealt with as follows.
\begin{itemize}
\item[(1)]
We claim that, for $l=0,1,$ and for any $\Psi\geq 0$, there exists a constant $c(\Psi)>0$ such that
\begin{equation}\label{eq:estKernel}
\int_{\mathbb{R}^{3}}\langle v\rangle^{\Psi}|\partial_{z_{1}}^{l}K(v,z)|\ d^3 v\leq c(\Psi) \langle z\rangle^{\Psi+2},\ \ \int_{\mathbb{R}^{3}}\langle z\rangle^{\Psi}|\partial_{z_{1}}^{l}K(z,u)|\ d^3 z\leq C(\Psi) \langle u\rangle^{\Psi+2}.
\end{equation}
\item[(2)] By direct computation,
\begin{equation}
|\partial_{z}^{l}\frac{1}{t[\partial_{z_1} \nu(z)+in_1]}|\lesssim \frac{1}{|{\bf{n}}|t}\ \text{for}\ l=0,1.
\end{equation}
\end{itemize}
These bounds and the fact that $e^{-t\nu}\lesssim e^{-\Lambda t}$ (see \eqref{eq:LambdaT}) imply that
$$\int_{\mathbb{R}^{3}\times \mathbb{R}^{3}} \langle v\rangle^{\Psi}|K_{t}^{({\bf{n}})}(v,u)g(u)|\ d^3 u\lesssim \frac{e^{-\Lambda t}}{|{\bf{n}}|t}\|\langle v\rangle^{\Psi+3}g\|_{L^{1}}.$$ To remove the non-integrable singularity in the upper bound at $t=0$, we use a straightforward estimate derived from the definition of $K_{t}^{({\bf{n}})}$ to obtain
$$\int_{\mathbb{R}^{3}\times \mathbb{R}^{3}} \langle v\rangle^{\Psi}|K_{t}^{({\bf{n}})}(v,u)g(u)|\ d^3 u\leq C(\Psi) e^{-\Lambda t}\|\langle v\rangle^{\Psi+3}g\|_{L^{1}}.
$$ Combination of these two estimates yields \eqref{eq:oscilate}.

We are left with proving \eqref{eq:estKernel}. In the next we focus on proving \eqref{eq:estKernel} when $l=1$, the case $l=0$ is easier, hence omitted.
By direct computation we find that $$
\begin{array}{lll}
|\partial_{z_{1}}K(v,z)|&\lesssim &  |\partial_{z_{1}}|z-v|^{-1}e^{-\frac{|(z-v)\cdot v|^2}{|z-v|^2}}|+|\partial_{z_{1}}|z-v|e^{-|v|^2}| 
\end{array}
$$
and, similarly, that $$|\partial_{z_{1}}K(z,u)|\lesssim  |\partial_{z_{1}}|z-u|^{-1}e^{-\frac{|(z-u)\cdot z|^2}{|z-u|^2}}|+ |\partial_{z_{1}}|z-u|e^{-|z|^2}|.$$
Among the various terms we only study the most difficult one, namely $\partial_{z_{1}}\tilde{K}(v,z)$, where $\tilde{K}(v,z)$ is defined by $$\tilde{K}(v,z):=|z-v|^{-1}e^{-\frac{|(z-v)\cdot v|^2}{|z-v|^2}}.$$
By direct computation $$|\partial_{z_{1}}\tilde{K}(v,z)|\lesssim \frac{1+|v_1|}{|v-z|^{2}}  e^{-\frac{1}{2}\frac{|(z-v)\cdot v|^2}{|z-v|^2}}.$$
To complete our estimate we divide the set $(v,z)\in \mathbb{R}^{3}\times \mathbb{R}^{3}$ into two subsets defined by $|v|\leq 10|z|$ and $|v|>10|z|,$ respectively. In the first subset we have that
$$|\partial_{z_{1}}\tilde{K}(v,z)|\lesssim \frac{1}{|v-z|^2}(|v|+1)\leq \frac{10(|z|+1)}{|v-z|^2},$$ and hence
\begin{equation}
\int_{|v|\leq 8|z|}\langle v\rangle^{\Psi}|\partial_{z_{1}}\tilde{K}(v,z)|\ d^3 v\leq 10(1+|z|)^{\Psi+1} \int_{|v|\leq 10|z|}\frac{1}{|v-z|^2}\ d^3 v\lesssim (1+|z|)^{\Psi+2}.
\end{equation}
In the second subset we have that $z-v\approx -v$, which implies that $\frac{|(z-v)\cdot v|}{|z-v|}\geq \frac{1}{2}|v|$. Thus, $$|\partial_{z_{1}}\tilde{K}(v,z)|\leq \frac{1+|v|}{|v|^2}e^{-\frac{1}{8}|v|^2}.$$ This obviously implies that
\begin{equation}
\int_{|v|\geq 10|z|}\langle v\rangle^{\Psi}|\partial_{z_{1}}\tilde{K}(v,z)|\ d^3 v\lesssim  \int_{|v|\geq 10|z|}\langle v\rangle^{\Psi}\frac{1+|v|}{|v|^2}\ e^{-\frac{1}{8}|v|^2} d^3 v\lesssim 1.
\end{equation}

By such estimates the proof of \eqref{eq:estKernel} can be easily completed.
\end{proof}

\section{Proof of Lemma \ref{LM:roughEst}}\label{sec:RoughEst}

\begin{proof}
Before we study the linear unbounded operator 
\begin{align}
L_{\bf{n}}:=\nu(v)+iv\cdot {\bf{n}}+K,\ {\bf{n}}\in \mathbb{Z}^3,
\end{align} mapping $\langle v\rangle^{-m}L^{1}(\mathbb{R}^3)$ into the same space, we start with studying $L_{\bf{n}}$, mapping $M^{\frac{1}{2}}L^2(\mathbb{R}^3)$ into itself.
Here the definitions of the spaces $\langle v\rangle^{-m}L^{1}(\mathbb{R}^3)$ and $M^{\frac{1}{2}}L^2(\mathbb{R}^3)$ are
\begin{align}
\langle v\rangle^{-m}L^{1}(\mathbb{R}^3):=\Big\{f:\ \mathbb{R}^3\rightarrow \mathbb{C}| \ \|\langle v\rangle^{m}f\|_{L^1}<\infty\Big\}
\end{align} and
\begin{align}
M^{\frac{1}{2}}L^{2}(\mathbb{R}^3):=\Big\{f:\ \mathbb{R}^3\rightarrow \mathbb{C}| \ \|M^{-\frac{1}{2}}f\|_{L^2}<\infty\Big\}.
\end{align} Here recall that $M=M_{\frac{1}{2},0}$ is the Maxwellian solution, see \eqref{eq:abbre}.

Denote the spectrum of the unbounded linear operator $L_{\bf{n}}$, mapping $M^{\frac{1}{2}}L^{2}(\mathbb{R}^3)$ into itself, by $\sigma(L_{\bf{n}}).$ Then since $K$ is a compact operator in the chosen space, we have that
\begin{align}
\sigma(L_{\bf{n}})=\sigma_{d}(L_{\bf{n}})\cup \sigma_{ess}(L_{\bf{n}}).
\end{align}
Recall that $L_{\bf{n}}$ is related to $L:=\nu(v)+v\cdot \nabla_x+ K$ by the fact that
\begin{align*}
L e^{i {\bf{n}}\cdot x}f=e^{i {\bf{n}}\cdot x}L_{\bf{n}}f. 
\end{align*} Hence if $f$ is an eigenvector for $L_{\bf{n}}$ in the space $M^{\frac{1}{2}}L^2(\mathbb{R}^3)$, then $e^{i {\bf{n}}\cdot x}f$ is an eigenvector for $L$ in the $M^{\frac{1}{2}}L^2(\mathbb{R}^3\times \mathbb{T}^3)$ space, with the same eigenvalue.

By this we have the following results.
\begin{lemma}\label{LM:subspace}
If $f:\ \mathbb{R}^3\rightarrow \mathbb{C}$ is an eigenvector for $L_{\bf{n}}$ in the space $M^{\frac{1}{2}}L^2(\mathbb{R}^3)$, then $e^{i {\bf{n}}\cdot x}f$ is an eigenvector for $L$ in the $M^{\frac{1}{2}}L^2(\mathbb{R}^3\times \mathbb{T}^3)$ space, with the same eigenvalue.

The set of eigenvalues of $L_{\bf{n}}: \ M^{\frac{1}{2}}L^2(\mathbb{R}^3)\rightarrow  M^{\frac{1}{2}}L^2(\mathbb{R}^3)$, is a subset of that of $L:\ M^{\frac{1}{2}}L^2(\mathbb{R}^3\times \mathbb{T}^3)\rightarrow M^{\frac{1}{2}}L^2(\mathbb{R}^3\times \mathbb{T}^3)$.

Moreover if $f\in M^{\frac{1}{2}}L^2(\mathbb{R}^3\times \mathbb{T}^3)$ is an eigenvector of $L$ with eigenvalue $\lambda$, and $\langle f,\ e^{i{\bf{n}}\cdot x}\rangle_{\mathbb{T}^3}=\int_{\mathbb{T}^3} f(v,x) e^{i {\bf{n}}\cdot x}\ dx\not=0$, then $\langle f,\ e^{i{\bf{n}}\cdot x}\rangle_{\mathbb{T}^3}\in  M^{\frac{1}{2}}L^2(\mathbb{R}^3)$ is an eigenvector of $L_{\bf{n}}$ with eigenvalue $\lambda.$
\end{lemma}

To locate the essential spectrum in the space $M^{\frac{1}{2}}L^2(\mathbb{R}^3)$, we use that $K$ is compact to find
\begin{align}
\sigma_{ess}(L_{\bf{n}})=\{\nu(v)+iv\cdot {\bf{n}}| v\in \mathbb{R}^3\}.
\end{align}
By known results, see \cite{MR1313028, CerIllPulv, MR2301289, YGuo2002, YGuo2003}, and Lemma \ref{LM:subspace}, there exist sets $A_{\bf{n}}\subset \mathbb{C}$ such that for ${\bf{n}}\not=(0,0,0)$
\begin{align}
\sigma_{d}(L_{{\bf{n}}})=A_{\bf{n}},\ \text{and}\ \sigma_{ess}(L_{{\bf{n}}})=\Big\{\nu(v)+i{\bf{n}}\cdot v\ |v\in \mathbb{R}^3\Big\};\label{eq:ess}
\end{align} and for ${\bf{n}}=(0,0,0)$
\begin{align}
\sigma_{d}(L_{{\bf{n}}})=\{0\}\cup A_{\bf{n}}, \  \text{and}\ \sigma_{ess}(L_{{\bf{n}}})=\{\nu(v)\ |v\in \mathbb{R}^3\}.\label{eq:ess0}
\end{align} Here the sets $A_{\bf{n}}$ keep a uniform distance from the imaginary axis, specifically, there exists a positive constant $\Lambda$ satisfying
\begin{align}
\Lambda\in \Big(0,\displaystyle\inf_{v\in \mathbb{R}^3}\nu(v)\Big)\label{def:Lambda}
\end{align} such that
\begin{align}
Re\lambda\geq \Lambda>0\ \text{if}\ \lambda\in \cup_{\bf{n}\in \mathbb{Z}^3}A_{\bf{n}}.\label{eq:gap}
\end{align}

\begin{figure}[htb!]
\centering%
\includegraphics[width=8cm]{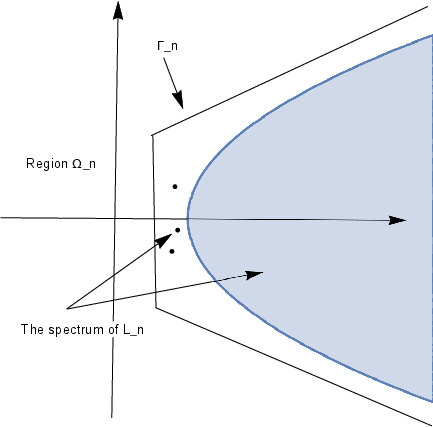}
\caption{The spectrum of $L_{{\bf{n}}}$, the curve $\Gamma_{{\bf{n}}},$ and the region $\Omega_{{\bf{n}}}$    }
\label{fig:Ln}
\end{figure}

In what follows we study $L_{{\bf{n}}},\ {\bf{n}}\not=(0,0,0)$. For ${\bf{n}}=(0,0,0),$ the analysis is similar except that $0$ is an eigenvalue.

Based on the informations about the spectrum of $L_{\bf{n}}$ in \eqref{eq:ess}, we have the following results.
For any ${\bf{n}}\in \mathbb{Z}^{3}\backslash{(0,0,0)},$ we define a curve $\Gamma_{{\bf{n}}}$ (see Figure \ref{fig:Ln}) to encircle the spectrum of $L_{{\bf{n}}}$,
\begin{equation}\label{eq:difCurve}
\Gamma_{{\bf{n}}}:=\Gamma_{1}({\bf{n}})\cup \Gamma_{2}({\bf{n}})\cup \Gamma_{3}({\bf{n}})
\end{equation} with
$$\Gamma_{1}({\bf{n}}):=\Big\{\Theta+i\beta |\ \beta \in [-\Psi(|{\bf{n}}|+1),\ \Psi(|{\bf{n}}|+1)]\Big\};$$
$$\Gamma_{2}({\bf{n}}):=\Big\{\Theta+i(|{\bf{n}}|+1)\Psi+\beta+i \Psi \beta(|{\bf{n}}|+1) ,\ \beta\geq 0\Big\};$$
$$\Gamma_{3}({\bf{n}}):=\Big\{\Theta-i(|{\bf{n}}|+1)\Psi+\beta-i \Psi \beta(|{\bf{n}}|+1) ,\ \beta\geq 0\Big\}.$$ Here $\Psi$ is a large positive constant to be chosen later, see \eqref{eq:ReasonCurve2}, Lemma \ref{LM:distanceContour} and \eqref{eq:ReasonCurve} below; $\Theta>0$ can be any constant in $(0,\ \frac{1}{2}\Lambda)$, with $\Lambda$ being the same one in \eqref{eq:gap}.
Moreover, we define $\Omega_{{\bf{n}}}$ to be the complement of the region encircled by the curve $\Gamma_{{\bf{n}}};$ see Figure \ref{fig:Ln}.

For the multiplication operator $\nu+i{\bf{n}}\cdot v-\zeta,$ if the constant $\Psi$ in the definition of the curves $\Gamma_{k,{\bf{n}}}, \ k=0,1,2,$ in \eqref{eq:difCurve}, are sufficiently large, then there exists a constant $C$ such that for any $\zeta\in\Gamma_{{\bf{n}}}$
\begin{equation}\label{eq:ReasonCurve2}
|\nu+i {\bf{n}}\cdot v-\zeta|^{-1} \leq C(1+|v|+|{\bf{n}}\cdot v|)^{-1}.
\end{equation}
It is straightforward, but a little tedious to verify this. Details are omitted.

Now we develop a convenient representation for $e^{-tL_{{\bf{n}}}}$. By standard technique of functional calculus we have that, for any bounded linear operator $A: L^2\rightarrow L^2$, and for any time $t$
\begin{align}
 e^{tA}=\frac{1}{2\pi i}\oint_{\gamma} e^{t\zeta} [\zeta-A]^{-1} \ d\zeta,\label{eq:functionalCal}
\end{align} where the curve $\gamma$ is closed and encircles the spectrum of $A.$ The identity in \eqref{eq:functionalCal} does not apply directly here since
the operator $L_{\bf{n}}$ is unbounded. However similar problems were considered in many literatures, for example, in Theorem 5.4 of our paper \cite{MR2187292} where a nonself-adjoint Schr\"odinger operator on a matrix was considered. Similar techniques apply to the present situation. Since this is tedious, but not difficult, we choose to skip the details here.

Thus, if $\Psi$ in \eqref{eq:difCurve} is large enough, then for ${\bf{n}}\not= (0,0,0)$ and for any $g\in M^{\frac{1}{2}}L^2$, we have that
\begin{equation}\label{eq:IntContour}
e^{-tL_{{\bf{n}}}}g=\frac{1}{2\pi i}\oint_{\Gamma_{{\bf{n}}}} e^{-t\zeta}[\zeta-L_{{\bf{n}}}]^{-1} \ d\zeta\ g,
\end{equation} recall that $L_{{\bf{n}}}-\zeta=\nu+i {\bf{n}}\cdot v-\zeta+K$. To see $[\zeta-L_{{\bf{n}}}]^{-1}=[\nu+i {\bf{n}}\cdot v-\zeta]^{-1} [1+K(\nu+i {\bf{n}}\cdot v-\zeta)^{-1}]^{-1} $ is well defined and uniformly bounded, we use the key fact that the operator $K:\ M^{\frac{1}{2}}L^2\rightarrow M^{\frac{1}{2}}L^2$ is compact, and discuss two different cases:

\begin{itemize}
\item[(a)]
If $|\zeta|\gg 1$ or $|{\bf{n}}|\gg 1$, then this together with \eqref{eq:ReasonCurve2} implies that $|\nu+i {\bf{n}}\cdot v-\zeta|\gg 1$ everywhere except for a small set, 
this makes the operator $K(\nu+i {\bf{n}}\cdot v-\zeta)^{-1}$ small since the operator $K$ is compact, hence $[\zeta-L_{{\bf{n}}}]^{-1}$ is uniformly well defined. 
To have quantitative version of $|\nu+i {\bf{n}}\cdot v-\zeta|\gg 1$ everywhere except for a small set is easy, but tedious. Moreover a sophisticated version of this techniques will be used to prove Propositions \ref{prop:roughB} and \ref{prop:largeN} below, which are more involved. Hence we choose to skip the details here.
\item[(b)]
If $|\zeta|=\mathcal{O}(1)$ and $|{\bf{n}}|=\mathcal{O}(1),$ then since we require $\Psi$ in \eqref{eq:difCurve} to be sufficiently large, we have that $\zeta\in \Gamma_1({\bf{n}})$. 
Then the uniformity is implied by Lemma \ref{LM:subspace} and the spectrum of $L$.

\end{itemize}
Motivated by Cook's method, see \cite{RSIII}, we consider the identity \eqref{eq:IntContour} in the space $\langle v\rangle^{-m}L^1(\mathbb{R}^3)$, defined as
\begin{align}
\langle v\rangle^{-m}L^1:= \Big\{ f:\ \mathbb{R}^3\rightarrow \mathbb{C}\ \Big| \ \|\langle v\rangle^{m}f\|_{L^1}<\infty \Big\}.
\end{align}
By the wellposedness, to be proved in Section \ref{sec:Local} below, we have that for any time $t\geq 0,$ $e^{-tL_{{\bf{n}}}}g\in \langle v\rangle^{-m}L^1$ if $g\in \langle v\rangle^{-m}L^1.$

For the term on the right hand side of \eqref{eq:IntContour}, the following lemma provides an important estimate.
\begin{lemma}\label{LM:distanceContour} There exists a large constant $Y$ such that if $m\geq Y$,
and if the positive constant $\Psi$ in \eqref{eq:difCurve} is sufficiently large, then there exists a constant $C=C(m)$ independent of ${\bf{n}}$ and $\zeta\in \Gamma_{{\bf{n}}}$ such that, 
for any point $\zeta\in \Gamma_{{\bf{n}}}$ and ${\bf{n}}\in \mathbb{Z}^{3}\backslash\{(0,0,0)\},$ we have $$\|(L_{{\bf{n}}}-\zeta)^{-1}\|_{\langle v\rangle^{-m}L^{1}\rightarrow \langle v\rangle^{-m} L^{1}}\leq C.$$
\end{lemma}
This lemma will be proven in section \ref{sec:distanceContour}.

Applying Lemma \ref{LM:distanceContour} to \eqref{eq:IntContour}, we obtain that, for $g\in \langle v\rangle^{-m}L^1\cap M^{\frac{1}{2}}L^2,$
$$
\|e^{-tL_{{\bf{n}}}}g\|_{\langle v\rangle^{-m}L^1}\lesssim  \int_{\zeta\in \Gamma_1({\bf{n}})\cup \Gamma_{2}({\bf{n}})\cup \Gamma_{3}({\bf{n}})} e^{-t Re\ \zeta} \ |d\zeta|\ \|g\|_{ \langle v\rangle^{-m}L^1}
$$ By the definition of $\Gamma_{1}({\bf{n}})$, it is easy to see that $$\int_{\zeta\in \Gamma_{1}} e^{-\Theta t} |d\zeta|\lesssim e^{-\Theta t} (|{\bf{n}}|+1).$$
Similarly, the definitions of $\Gamma_{2}({\bf{n}})$ and $\Gamma_{3}({\bf{n}})$ imply that for any $t\geq 1,$  $$\int_{\zeta\in \Gamma_{2}({\bf{n}})\cup \Gamma_{3}({\bf{n}})} e^{-t Re \zeta} \ |d\zeta| \lesssim (1+|{\bf{n}}|)\int_{\Theta}^{\infty} e^{-t\sigma} d\sigma\lesssim e^{-\Theta t} (1+|{\bf{n}}|).$$

Collecting the estimates above and using the fact that $\langle v\rangle^{-m}L^1\cap M^{\frac{1}{2}}L^2$ is dense in $\langle v\rangle^{-m}L^1$, we prove \eqref{eq:est000}, for $t\geq 1.$

The proof will be complete if we can show that the propagator $e^{-tL_{{\bf{n}}}}$ is bounded on $L^{1}(\mathbb{R}^3)$ when $t\in[0,1].$ To prove this, we establish the local wellposedness of the linear equation
\begin{align}
\begin{split}\label{eq:lin}
\partial_{t}g=&[-\nu-i{\bf{n}}\cdot v-K]g,\\
g(v,0)=&g_0(v),
\end{split}
\end{align} in Appendix \ref{sec:Local} below, which shows that, there exists a constant $C$, independent of $\bf{n}$, s.t. \eqref{eq:lin} has a unique solution 
in the time interval $[0,\ 1]$ and it satisfies the estimate $$\|\langle v\rangle^{m}g(\cdot, t)\|_{L^1}\leq C \|\langle v\rangle^{m}g_0\|_{L^1}.$$

This completes the proof of \eqref{eq:est000}.

The proof of \eqref{eq:estNot0} is almost identical, the only difference is that the operator $L_0$ has an isolated eigenvalue $0$. Hence we choose to skip the details here.

This completes the proof of Lemma \ref{LM:roughEst}.
\end{proof}

\section{Proof of Lemma \ref{LM:distanceContour}}\label{sec:distanceContour}
As stated in Lemma \ref{LM:distanceContour}, we need $m$ sufficiently large to make certain constants sufficiently small. 
In the rest of the paper, we keep track all the constants related to $m.$ For the purpose of notation, in what follows we use $a\lesssim b$ to signify that
\begin{align}
a\leq Cb\label{eq:CInM}
\end{align} with $C$ being a fixed constant, independent of $m.$

We start by simplifying the arguments in Lemma \ref{LM:distanceContour}. Using the definitions of the operators $L_{{\bf{n}}}$, ${\bf{n}}\in \mathbb{Z}^3$, in \eqref{eq:observation}, $K$ in \eqref{eq:difK}, and $\nu$ in \eqref{eq:difNu} we find that $$L_{{\bf{n}}}=\nu+K+i{\bf{n}}\cdot v .$$
In order to prove the uniform invertibility of $L_{{\bf{n}}}-\zeta,\ \zeta\in \Gamma_{{\bf{n}}},$ we claim that it suffices to prove this property for $1-K_{\zeta,{\bf{n}}}$ with $K_{\zeta,{\bf{n}}}$ defined by
\begin{equation}\label{eq:difKsn}
K_{\zeta,{\bf{n}}}:=K(\nu+i{\bf{n}} \cdot v -\zeta)^{-1}.
\end{equation} To see that, rewrite $L_{{\bf{n}}}-\zeta$ as
\begin{align}\label{eq:rewriteLN}
L_{{\bf{n}}}-\zeta=[1+K_{\zeta,{\bf{n}}}](\nu+i{\bf{n}}\cdot v-\zeta).
\end{align} 
For the multiplication operator $\nu+i{\bf{n}}\cdot v-\zeta,$ if the constant $\Psi$ in the definition of the curves $\Gamma_{k,{\bf{n}}}, \ k=0,1,2,$ in \eqref{eq:difCurve}, are sufficiently large, then there exists a constant $C$ such that for any $\zeta\in\Gamma_{{\bf{n}}}$
\begin{equation}\label{eq:ReasonCurve}
|\nu+i {\bf{n}}\cdot v-\zeta|^{-1} \leq C(1+|v|+|{\bf{n}}\cdot v|)^{-1}.
\end{equation}
It is straightforward, but a little tedious to verify this. Details are omitted.

Now we study the linear operator $1+K_{\zeta,{\bf{n}}},$ the result is the following: recall the definition of space 
\begin{align}
\langle v\rangle^{-m}L^1:= \Big\{ f:\ \mathbb{R}^3\rightarrow \mathbb{C}\ \Big| \ \|\langle v\rangle^{m}f\|_{L^1}<\infty \Big\}.
\end{align}
\begin{lemma}\label{LM:compactLm}
Suppose that $m>0$ is sufficiently large. Then
for any point $\zeta\in \Gamma_{{\bf{n}}}$ and ${\bf{n}}\in \mathbb{Z}^{3}\backslash\{(0,0,0)\},$ we have that $1+K_{\zeta,{\bf{n}}}:\ \langle v\rangle^{-m}L^{1}\rightarrow \langle v\rangle^{-m}L^{1}$ is invertible; its inverse satisfies the estimate 
\begin{align}
\|(1+K_{\zeta,{\bf{n}}})^{-1}\|_{\langle v\rangle^{-m}L^{1}\rightarrow \langle v\rangle^{-m}L^{1}}\leq C(m),
\end{align} where the constant $C(m)$ is independent of ${\bf{n}}$ and $\zeta$.
\end{lemma}
This will be proven after presenting the key ideas.

The results above complete the proof of Lemma \ref{LM:distanceContour}, assuming that Lemma \ref{LM:compactLm} holds.


Next we prove Lemma \ref{LM:compactLm}. Here an obvious difficulty is that the set $\{({\bf{n}},\zeta)\ |\ {\bf{n}}\in \mathbb{Z}^3,\ \zeta\in \Gamma_{\bf{n}}\}$ is not compact, this makes it hard to find an uniform bound.
To overcome the difficulty we divide the set into three subsets and apply different techniques:
specifically, for some large constants $N$ and $X$,
\begin{itemize}
\item[(1)] $|{\bf{n}}|> N$, 
\item[(2)] $|{\bf{n}}|\leq N$, and $|\zeta|> X$, 
\item[(3)] $|{\bf{n}}|\leq N$, and $|\zeta|\leq X$.
\end{itemize}

Next we look for the constants $N$ and $X$ by Propositions \ref{prop:roughB} and \ref{prop:largeN} below. 
For the rest, i.e. in the compact regime (3), we apply Proposition \ref{prop:compact}.

Recall the constant $\Upsilon_{m}:=\Upsilon_{m,\frac{1}{2}} $ defined in \eqref{eq:mK1}, and that $\Upsilon_m\rightarrow \infty$ as $m\rightarrow \infty.$
\begin{proposition}\label{prop:roughB}
There exists a constant $Y>0$, such that if $m\geq Y$, and if $\zeta\in \Gamma_{\bf{n}}$ is large enough to satisfy
\begin{align}
|\zeta|\geq (1+|{\bf{n}}|^2)[\Upsilon_m+m]^4,
\end{align} then we have
\begin{align}
 \|\langle v\rangle^{m} (1
 +K_{\zeta,{\bf{n}}})^{-1}\langle v\rangle^{-m}\|_{L^1\rightarrow L^1}\leq 2 .
\end{align} 

\end{proposition}

The proposition will be proved in subsection \ref{subsec:roughB}.

Here we present some basic ideas. 
By the fact $K_{\zeta,\bf{n}}=(K_1-K_2-K_3)\ (\nu+i{\bf{n}}\cdot v-\zeta)^{-1},$ we have that in the region $|v|\leq |\zeta|^{\frac{1}{2}}$, 
\begin{align}
 |(\nu+i{\bf{n}}\cdot v-\zeta)^{-1}|\leq \langle v\rangle^{-1} |\zeta|^{-\frac{1}{2}}.
\end{align} The smallness is rendered by that $|\zeta|\gg 1.$
For the region $|v|>|\zeta|^{\frac{1}{2}}\gg 1$ the integral kernel $K_{\zeta,{\bf{n}}}$, which is localized in some sense, makes the contribution of this part to be small. Recall that in certain weighted $L^2$ space, $K_{l},\ l=1,2,3,$ are compact.

Next, we state the second result. 

\begin{proposition}\label{prop:largeN}
There exist a large constant $Y>0$ and a small one $\epsilon>0$, such that if $m\geq Y$, 
and if $|\bf{n}|$ is large enough to make 
\begin{align}
\Upsilon_{m}^2|{\bf{n}}|^{-\frac{1}{5}}+C(m)|{\bf{n}}|^{-\frac{1}{10}}\leq \epsilon
\end{align} for some fixed constant $C(m),$ 
then for any $\zeta\in \Gamma_{\bf{n}}$,
\begin{align}
\|\langle v\rangle^{m}(1+ K_{\zeta,{\bf{n}}})^{-1}\langle v\rangle^{-m}\|_{L^{1}\rightarrow L^1}\lesssim \Upsilon_m .\label{eq:largeN}
\end{align}
\end{proposition}
The proposition will be proved in subsection \ref{sub:LargeN}. 
The basic ideas in proving Proposition \ref{prop:largeN} are easy. Recall that by definition $$K_{\zeta,\bf{n}}=(K_1-K_2-K_3)\ (\nu+i\bf{n}\cdot v -\zeta)^{-1}.$$ 
When $|\bf{n}|$ is large, the purely imaginary part of $\nu+i\bf{n}\cdot v-\zeta$, which is ${\bf{n}}\cdot v-Im\zeta$, is large except for a ``small" set, for example a neighborhood of the sets $v\perp {\bf{n}}$ and $v=0$. This will render $K_{\zeta,\bf{n}}$ small except for a small set. For the small set, the integral kernels of $K_{l}$, which is bounded and continuous, make the contribution small.

The following result is an estimate for each fixed $({\bf{n}},\zeta)$ in the set $\{({\bf{n}},\zeta)\ |\ {\bf{n}}\in \mathbb{Z}^3,\ \zeta\in \Gamma_{\bf{n}}\}$. Recall the constant $\Lambda$ from \eqref{eq:gap},
\begin{proposition}\label{prop:compact}
There exists a constant $Y$ such that if $m\geq Y$, then for each fixed ${\bf{n}}\in \mathbb{Z}^3$ and $\zeta\in \Gamma_{\bf{n}}$, there exists some constant $C_{\bf{n},\zeta}>0$ such that 
\begin{align}
\|\langle v\rangle^{m}(1+K_{\zeta,{\bf{n}}})^{-1}\ g\|_{L^{1}}\leq C_{\bf{n},\zeta}\|\langle v\rangle^{m}g\|_{L^{1}}.\label{eq:invert}
\end{align} 
\end{proposition}
The proof will be in Subsection \ref{sec:matchingArgu}, here we use some construction and ideas from \cite{Mouhot2006, Mouhot2010}, see also \cite{Ark1988,Wenn1993}. However our proof is self-contained, and is more direct.


Based on Propositions \ref{prop:roughB}, \ref{prop:largeN} and \ref{prop:compact}, we ready to prove Lemma \ref{LM:compactLm}.
\begin{proof}
We start with choosing $N$ and $X$, to define the three regimes listed before Proposition \ref{prop:roughB}.

Let $m\geq Y$, with $Y$ large enough to make Proposition \ref{prop:roughB}, \ref{prop:largeN} and \ref{prop:compact} applicable. 

Then choose $N\in \mathbb{N}$ large enough to make $\Upsilon_m^2 |N|^{-\frac{1}{5}}+C(m)|N|^{-\frac{1}{10}}\leq \epsilon$, then by Propositions \ref{prop:largeN}, for any ${\bf{n}}$ with $|{\bf{n}}|\geq N$,
\begin{align}
\Big\|\langle v\rangle^{m} [1+K_{\zeta,{\bf{n}}}]^{-1}\langle v\rangle^{-m}\Big\|_{L^1\rightarrow L^1}\lesssim \Upsilon_{m}.
\end{align}
After choosing $N$, we choose $X$ as 
\begin{align}
X:=(1+|N|^2)[\Upsilon_m+m]^4,
\end{align} so that for any $|{\bf{n}}|\leq N$ and $\zeta\in \Gamma_{{\bf{n}}}$ satisfying $|\zeta|\geq X$, Proposition \ref{prop:roughB} applies.

What is left is the regime where $|{\bf{n}}|\leq N$ and $|\zeta|\leq X$, here we apply Proposition \ref{prop:compact}. The constant $C_{\bf{n},\zeta}$ in \eqref{eq:invert} is uniformly bounded since here the considered regime $|{\bf{n}}|\leq N$ and $|\zeta|\leq X$ is compact.

Collecting the estimates above, we prove Lemma \ref{LM:compactLm}.

\end{proof}

In the rest of this section, we prove Propositions \ref{prop:roughB}, \ref{prop:largeN} and \ref{prop:compact}. Upon completion of the work, we realize that, by reading known works such as \cite{BoGamPan04, Mouhot2010}, some of the ideas to be used in proving Propositions \ref{prop:roughB}, \ref{prop:largeN} and \ref{prop:compact} were in proving Povzner's inequality, see also \cite{ MR1478067, MR1461113, Mouhot2010, Mouhot2006}. However in general our proof is more direct, and self-contained.

Before the proof we define a small constant.

Recall the definitions of operators $K_{l},\ l=1,2,3,$ in \eqref{eq:difK2}. Define a new quantity $\delta_{m,0}$ by
\begin{align}
\delta_{m,0}:=\sum_{l=1}^{3}\|\chi_{> m}\langle v\rangle^{m} K_{l} \langle v\rangle^{-m-1}\chi_{> m}\|_{L^1\rightarrow L^1}.\label{eq:difDeltaM0}
\end{align}
Here $\chi_{>m}$ is a Heaviside function defined as
\begin{align}\label{eq:cutoff}
\chi_{> m}(v)=\left[
\begin{array}{cc}
1\ &\ \text{if}\ |v|> m\\
0\ &\ \text{otherwise.}
\end{array}
\right.
\end{align}
The result is
\begin{lemma}\label{LM:Control}
The quantity $\delta_{m,0}$ satisfies the following estimate
\begin{align}
\delta_{m,0}\rightarrow 0\ \text{as} \ m\rightarrow +\infty.
\end{align} 
\end{lemma}
\begin{proof}

It is easy to estimate the $K_1-$term,
by the rapidly decaying factor $e^{-|v|^2}$ in the integral kernel of $K_1$ and that $\lim_{m\rightarrow \infty}\int_{|v|> m}e^{-|v|^2}\ dv= 0$. Compute directly to have, for any nonzero function $f,$
\begin{align}
\frac{\|\chi_{> m} \langle v\rangle^{m}K_1 \langle v\rangle^{-m-1} \chi_{> m}f\|_{L^1}}{\|f\|_{L^1}}\rightarrow 0 \ \text{as}\ m\rightarrow \infty.\label{eq:k1del01}
\end{align}

Now we start estimating the $K_2$-terms by casting the expression into a convenient form.
For $\omega\in \mathbb{S}^2$ in the definition of $K_{2}$, we look for an unitary rotation $U_{\omega}$ to make 
\begin{align}
U_{\omega}^* \omega=\left[
\begin{array}{cc}
1\\
0\\
0
\end{array}
\right].\label{eq:unitary}
\end{align}
For that purpose, since 
for any $\omega\in \mathbb{S}^2$, there exist unique $\theta\in [0,\ 2\pi)$ and $\gamma\in [-\frac{\pi}{2},\ \frac{\pi}{2})$ such that
\begin{align}
\omega=\left[
\begin{array}{c}
cos\theta \ cos\gamma\\
sin\theta\ cos\gamma\\
sin\gamma
\end{array}
\right],
\end{align} we choose the rotation $U_{\omega}^*$ in \eqref{eq:unitary} as 
\begin{align}\label{eq:rota}
U_{\omega}^*:=
\left[
\begin{array}{ccc}
cos\gamma &0 & sin\gamma \\
0&1&0\\
-sin\gamma & 0& cos\gamma
\end{array}
\right]
\left[
\begin{array}{ccc}
cos\theta & sin\theta &0\\
-sin\theta & cos\theta &0\\
0&0&1
\end{array}
\right].
\end{align} 

Insert this rotation into appropriate places of $K_2$ and $f$ and change variables $U_{\omega}^* u\rightarrow u,\ U_{\omega}^*v\rightarrow v$ to obtain
\begin{align}
\begin{split}\label{eq:smallIn2}
& \|\chi_{> m} \langle v\rangle^{m}K_2 \langle v\rangle^{-m-1} \chi_{> m}f\|_{L^1}\\
\leq &\int_{\mathbb{S}^2}\int_{R(m)}   e^{-|v_1|^2-|u_2|^2-|u_3|^2}
\frac{\langle v\rangle^{m} |u_1-v_1|}{(1+u_1^2+v_2^2+v_3^2)^{\frac{m+1}{2}}}\ |f|\big((u_1, v_2, v_3)\ U_{\omega}^*\big) \ d^3 u d^3 v d\omega\\
\lesssim &\int_{R(m)}   e^{-|v_1|^2}
\frac{| v|^{m} |u_1-v_1|}{(u_1^2+v_2^2+v_3^2)^{\frac{m+1}{2}}}\ \Big[ \int_{\mathbb{S}^2}|f|\big((u_1, v_2, v_3)\ U_{\omega}^*\big)\ d\omega\ \Big] \ d u_1 d^3 v 
\end{split}
\end{align}
where $R(m)$ is a set defined as
$$
R(m):=\Big\{(u,v)\in \mathbb{R}^3\times \mathbb{R}^3\ \Big|\ |v|> m, \
\sqrt{u_1^2+v_2^2+v_3^2}> m\Big\},
$$ and the $u_2-$ and $u_3-$variables are integrated out. To analyze in detail we adopt polar coordinate on $(u_1, v_2, v_3)$ by defining
\begin{align}
u_1=r \cos\alpha,\ v_2=r \sin\alpha \sin\beta, v_3=r \sin\alpha \cos\beta, \alpha\in [0,\pi], \beta\in [0,2\pi].
\end{align}

The corresponding part of \eqref{eq:smallIn2} becomes
\begin{align}
\begin{split}
D_1=&:\int_{ r,\ \sqrt{v_1^2+r^2\sin^2\alpha}> m} e^{-v_1^2} r^{-m-1} [v_1^2+r^2\sin^2\alpha]^{\frac{m}{2}} |v_1-r\cos\alpha| \sin\alpha \times\\
 &\int_{\mathbb{S}^2}|f|\big(r (\cos\alpha, \sin\alpha \sin\beta, \sin\alpha  \cos\beta)\ U_{\omega}^*\big)\ d\omega\ r^2 dr dv_1d\alpha d\beta.
\end{split}
\end{align} Compute directly to find
\begin{align}
D_1\lesssim \|\chi_{> m}f\|_{L^1} \sup_{r\geq m}r^{-m-1}\int_{\sqrt{v_1^2+r^2\sin^2\alpha}\geq m}e^{-v_1^2} r^{-m-1} [v_1^2+r^2\sin^2\alpha]^{\frac{m}{2}} |v_1-r\cos\alpha| \ dv_1.
\end{align}


Next we exploit the rapid decay of $e^{-|v_1|^2}$ by considering two integral regions: $|v_1|> m^{\frac{3}{4}}$ and $|v_1|\leq m^{\frac{3}{4}}$. 

For the first case $|v_1|> m^{\frac{3}{4}}\gg 1$, we compare $r^{-m-1} \Big| v_1^2+r^2 sin^2\alpha\Big|^{\frac{m}{2}} \Big(|r cos\alpha|+|v_1|\Big)$ to $|v_1|^{m+1}$ and find, since $r\geq m,$
\begin{align*}
|v_1|^{-m-1} r^{-m-1} \Big| v_1^2+r^2 sin^2\alpha\Big|^{\frac{m}{2}} \Big(|r cos\alpha|+|v_1|\Big)\leq& 2^{m}.
\end{align*}
Thus in the considered region
\begin{align}
&\sup_{r\geq m} \Big[r^{-m-1}\int_{\sqrt{v_1^2+r^2 sin^2\alpha}\geq m,\ |v_1|\geq m^{\frac{3}{4}}}   e^{-|v_1|^2}\Big| v_1^2+r^2 sin^2\alpha\Big|^{\frac{m}{2}} \Big(|r cos\alpha|+|v_1|\Big)
\ dv_1d\alpha\Big]\nonumber\\
\leq &2^{m}\int_{|v_1|\geq m^{\frac{3}{4}}} e^{-|v_1|^2} |v_1|^{m+1}dv_1\nonumber\\
= &3\cdot 2^{m-2}\int_{|u|\geq m} e^{-|u|^{\frac{3}{2}}} |u|^{\frac{3m
+2}{4}}du\leq
2\delta_{m,1}\label{eq:unbounded}
\end{align}
where, in the last step we changed variable $u=|v_1|^{\frac{4}{3}}$, and the constant $\delta_{m,1}$ is defined as
\begin{align}
\delta_{m,1}:=2^{m}\int_{z\geq m^{\frac{3}{4}}}   z^{2m} e^{-\frac{z^2}{4}} dz=3\cdot 2^{m-2}\int_{z\geq m} z^{\frac{3m}{2}-\frac{1}{4}} e^{-\frac{1}{4}z^{\frac{3}{2}}}\ dz.\label{eq:defDeltaM2}
\end{align} The fact that the function $e^{-\frac{1}{4}z^{\frac{3}{2}}}$ decays faster than $e^{-z}$ implies that $\frac{\delta_{m+1,2}}{\delta_{m,1}}\ll 1$ if $m$ is large. Consequently 
\begin{align}
\delta_{m,1}\rightarrow 0\ \text{as}\ m\rightarrow \infty. \label{eq:mInte}
\end{align}

Now we consider the region $|v_1|\leq m^{\frac{3}{4}}$. The condition $\sqrt{v_1^2+r^2 sin^2\alpha}\geq m$ implies that $r\sin\alpha\geq m\big(1+o(1)\big)$. This together with $r\geq m$ makes
\begin{align}
\begin{split}
&e^{-\frac{1}{2}|v_1|^2}r^{-m-1}\Big| v_1^2+r^2 sin^2\alpha\Big|^{\frac{m}{2}} \Big(|r cos\alpha|+|v_1|\Big)\\
\leq & 2\sin^{m}(\alpha) \ e^{-\frac{1}{2}|v_1|^2} \Big(1+\frac{|v_1|^2}{r^2 \sin^2\alpha}\Big)^{\frac{m}{2}}(|\cos\alpha|+m^{-\frac{1}{4}})\\
\leq & 2\sin^{m}(\alpha) \ e^{-\frac{1}{2}|v_1|^2} \Big(1+\frac{2|v_1|^2}{m^2}\Big)^{\frac{m}{2}}(|\cos\alpha|+m^{-\frac{1}{4}})\\
\leq & 2\sin^{m}(\alpha)(|\cos\alpha|+m^{-\frac{1}{4}}).
\end{split}
\end{align} where in the last step we used the inequality $ \Big(1+\frac{2|v_1|^2}{m^2}\Big)^{\frac{m}{2}}\leq e^{\frac{1}{2}|v_1|^2}
$, which is equivalent to $1+\frac{2|v_1|^2}{m^2}\leq e^{\frac{|v_1|^2}{m}}$, and the latter is proved by Taylor-expanding the exponential and using that $m\gg 1.$

This renders the integral small since
\begin{align}
\begin{split}
&\sup_{r\geq m} \Big[r^{-m-1}\int_{\sqrt{v_1^2+r^2 sin^2\alpha}\geq m,\ |v_1|\leq m^{\frac{3}{4}}}   e^{-|v_1|^2}\Big| v_1^2+r^2 sin^2\alpha\Big|^{\frac{m}{2}} \Big(|r cos\alpha|+|v_1|\Big)
\ dv_1\Big]\\
\lesssim &\sin^{m}(\alpha)(|\cos\alpha|+m^{-\frac{1}{4}})\rightarrow 0\ \text{as}\ m\rightarrow \infty.
\end{split}
\end{align}

This together with \eqref{eq:unbounded} and \eqref{eq:k1del01} imply the desired estimate for $K_2$ in Lemma \ref{LM:Control}.

Now we estimate the $K_3-$term.

To transform the expression into a convenient form, we follow the steps in \eqref{eq:smallIn2} to find
\begin{align}
\begin{split}\label{eq:smallIn20}
& \|\chi_{> m} \langle v\rangle^{m}K_3 \langle v\rangle^{-m-1} \chi_{> m}f\|_{L^1}\\
\lesssim &\int_{\mathbb{S}^2}\int_{R_2(m)}   e^{-|u_1|^2-|v_2|^2-|v_3|^2}
\frac{| v|^{m} |u_1-v_1|}{(v_1^2+u_2^2+u_3^2)^{\frac{m+1}{2}}}\ |f|\big((v_1, u_2, u_3)\ U_{\omega}^*\big) \ d^3 u d^3 v d\omega
\end{split}
\end{align}
where $R_2(m)$ is a set defined as
$$
R_2(m):=\Big\{(u,v)\in \mathbb{R}^3\times \mathbb{R}^3\ \Big|\ |v|> m, \
\sqrt{v_1^2+u_2^2+u_3^2}> m\Big\}.
$$

All the terms can be controlled by the same techniques used in analyzing the $K_2$-term, except the following one
\begin{align}
D_2:=\int_{\sqrt{v_1^2+u_2^2+u_3^2}> m}\frac{|v_1|^{m+1}}{(v_1^2+u_2^2+u_3^2)^{\frac{m+1}{2}}}\int_{\mathbb{S}^2}|f|\big((v_1, u_2, u_3)\ U_{\omega}^*\big) d\omega\ dv_1 du_2 du_3.
\end{align} The difficulty is that when $|v_1|\approx (v_1^2+u_2^2+u_3^2)^{\frac{1}{2}}$, it is hard to find smallness from $\frac{|v_1|^{m+1}}{(v_1^2+u_2^2+u_3^2)^{\frac{m+1}{2}}}$, which equals to 1 on a zero-measure set.
To overcome this we use the $\omega-$integral, thus free some integration variables so that by integrating them over a small region we find the smallness. To cast the expression into a convenient form and use the rotation $U_{\omega}^*$ in \eqref{eq:rota}, we
change the coordinate $$v_1=r\sin\beta \cos\alpha, \ v_2=r\cos\beta, \ v_3=r\sin\beta \sin\alpha,\ \beta\in [0,\pi],\ \alpha\in [0,2\pi]$$ to find
\begin{align}
D_2=\int_{r> m}(\sin\beta\ \cos\alpha)^{m+1}sin\beta\ r^2\int_{\mathbb{S}^2}|f|\big(r(\sin\beta \cos\alpha, \cos\beta,\ \sin\beta \sin\alpha)\ U_{\omega}^*\big) d\omega\ dr d\alpha  d\beta.
\end{align} In what follows we only consider the region $\Sigma$, defined as
\begin{align}
\Sigma:=\{(\beta,\alpha)\ \Big|\ |\cos\beta|,\ |\sin\alpha|\leq m^{-\frac{1}{4}}\},\label{twosmall}
\end{align} in the other region, the fact $(\sin\beta\ \cos\alpha)^{m+1}\rightarrow 0$ uniformly as $m\rightarrow \infty$ makes the integral small.
Observe that when $|\sin\beta \cos\alpha|=1$, or $\cos\beta=\sin\alpha=0$, the following identity holds
\begin{align}
2\int_{|x|> m} |f(x)|\ d^3x=\int_{r>m} r^2\int_{\mathbb{S}^2}|f|\big(r(\sin\beta \cos\alpha, \cos\beta,\ \sin\beta \sin\alpha)\ U_{\omega}^*\big) d\omega dr 
\end{align} Thus we approximate $\int_{r>m} r^2\int_{\mathbb{S}^2}|f|\big(r(\sin\beta \cos\alpha, \cos\beta,\ \sin\beta \sin\alpha)\ U_{\omega}^*\big) d\omega dr$ by $2\int_{|x|> m} |f(x)|\ d^3x$.

For that purpose we consider a factor of the integrand in $D_2$: recall the expression of $U_{\omega}^{*}$ in \eqref{eq:rota},
\begin{align}
\begin{split}
&r^2f(r(\sin\beta \cos\alpha, \cos\beta,\ \sin\beta \sin\alpha)\ U_{\omega}^*) d\omega dr\\
=&r^2f\Big(r\big(\sin\beta \cos(\gamma-\alpha), \cos\beta, \sin\beta \sin(\gamma-\alpha)\big)\left[
\begin{array}{ccc}
cos\theta & sin\theta &0\\
-sin\theta & cos\theta &0\\
0&0&1
\end{array}
\right]\Big) \cos\gamma\ d\gamma d\theta\ dr\\
=&E_1-E_2
\end{split} 
\end{align} 
where in the last step we approximate $\cos\gamma$ by $\cos(\gamma-\alpha)$, the terms $E_1$ and $E_2$ are produced after applying the identity $\cos\gamma=\cos(\gamma-\alpha+\alpha)=\cos(\gamma-\alpha) \cos(\alpha)-\sin(\gamma-\alpha) \sin(\alpha),$ and are defined as
\begin{align*}
E_1:=r^2 f\Big(r\big(\sin\beta \cos(\gamma-\alpha), \cos\beta, \sin\beta \sin(\gamma-\alpha)\big)\left[
\begin{array}{ccc}
cos\theta & sin\theta &0\\
-sin\theta & cos\theta &0\\
0&0&1
\end{array}
\right]\Big) \cos(\gamma-\alpha) \cos(\alpha) d\gamma d\theta\ dr
\end{align*}
and 
\begin{align*}
E_2:=r^2 f\Big(r\big(\sin\beta \cos(\gamma-\alpha), \cos\beta, \sin\beta \sin(\gamma-\alpha)\big)\left[
\begin{array}{ccc}
cos\theta & sin\theta &0\\
-sin\theta & cos\theta &0\\
0&0&1
\end{array}
\right]\Big) \sin(\gamma-\alpha) \sin(\alpha) d\gamma d\theta \ dr.
\end{align*} 

By the definitions of $E_1$ and $E_2$ we decompose $D_2$ into two parts
\begin{align}
D_2=D_{2,E_1}-D_{2,E_2}
\end{align} with the two terms naturally defined.

The key observation is after changing variables, with $\beta$ and $\alpha$ fixed, $$r\Big(\sin\beta \cos(\gamma-\alpha), \cos\beta, \sin\beta \sin(\gamma-\alpha)\Big)\left[
\begin{array}{ccc}
cos\theta & sin\theta &0\\
-sin\theta & cos\theta &0\\
0&0&1
\end{array}
\right]= x,$$
we have
\begin{align}
E_1=|f(x)|(1+o(1)) d^3 x.
\end{align} The calculation, specifically computing the determinant of Jacobian matrix, is easy but tedious, thus we skip this part.
Consequently
\begin{align}
D_{2,E_1}\lesssim \|\chi_{>m}f\|_{L^1} \int_{\Sigma} (\sin\beta\ |\cos\alpha|)^{m+1}sin\beta\ d\beta d\alpha\rightarrow 0\ \text{as}\ m\rightarrow \infty.
\end{align}
The smallness of $D_{2,E_2}$ is from $|\sin\alpha|\leq m^{-\frac{1}{4}}$, see \eqref{twosmall}, and different from estimating $D_{2,E_1}$, here we integrate $r$, $\alpha$ and $\beta$ first, then $\theta$ and $\gamma$. The observation 
$$\int_{r>m} r^2 sin\beta f\Big(r\big(\sin\beta \cos(\gamma-\alpha), \cos\beta, \sin\beta \sin(\gamma-\alpha)\big)\left[
\begin{array}{ccc}
cos\theta & sin\theta &0\\
-sin\theta & cos\theta &0\\
0&0&1
\end{array}
\right]\Big) dr d\beta d\alpha=\|\chi_{>m}f\|_{L^1},$$ directly implies that
\begin{align}
D_{2,E_2}\lesssim m^{-\frac{1}{4}}\|\chi_{>m}f\|_{L^1}.
\end{align}

Collect the estimates above to complete estimating the $K_3-$term.
\end{proof} 

\subsection{Proof of Proposition \ref{prop:roughB}}\label{subsec:roughB}

We start by casting the expression into a convenient form, by transforming the operator $1+K_{\zeta,\bf{n}}$ into a $2\times 2$ operator-valued matrix.
Let $\chi_{\leq 2m}$ be a Heaviside function
\begin{align}\label{eq:cutoff2}
\chi_{\leq 2m}(v)=\left[
\begin{array}{cc}
1\ &\ \text{if}\ |v|\leq 2m\\
0\ &\ \text{otherwise}
\end{array}
\right.
\end{align} and naturally $\chi_{> 2m}$ is defined as $$\chi_{> 2m}:=1-\chi_{\leq 2m}.$$

Decompose the $L^1(\mathbb{R}^3)$ space into a vector space, isometrically,
\begin{align}
L^1(\mathbb{R}^3)\rightarrow \left[
\begin{array}{lll}
\chi_{\leq 2m} L^1(\mathbb{R}^3)\\
\chi_{>2m}L^1(\mathbb{R}^3)
\end{array}
\right].
\end{align} 
Specifically, for any function $f$,  
\begin{align}
f\rightarrow \left[
\begin{array}{lll}
\chi_{\leq 2m} f\\
\chi_{>2m}f
\end{array}
\right]
\end{align}
with the norm defined as, 
\begin{align*}
\|\left[
\begin{array}{ll}
\chi_{\leq 2m}f\\
\chi_{> 2m}f
\end{array}
\right]\|_{L^1}:=\|\chi_{\leq 2m}f\|_{L^1}+\|\chi_{> 2m}f\|_{L^1}=\|f\|_{L^1}.
\end{align*}

Based on this, we convert the operator $1+K_{\zeta,{\bf{n}}}$ into an operator-valued $2\times 2$ matrix. Specifically, for any function $f$,
\begin{align}
 (1+K_{\zeta,{\bf{n}}})f\rightarrow (1+D)\left[
\begin{array}{lll}
\chi_{\leq 2m} f\\
\chi_{>2m}f
\end{array}
\right].
\end{align}
Here $D$ is an operator-valued $2\times 2$ matrix defined as
\begin{align}
D:=&\left[
\begin{array}{cc}
\chi_{\leq 2m} K_{\zeta,{\bf{n}}}\chi_{\leq 2m} & \chi_{\leq 2m} K_{\zeta,{\bf{n}}}\chi_{> 2m}\\
\chi_{> 2m} K_{\zeta,{\bf{n}}}\chi_{\leq 2m}&\chi_{> 2m} K_{\zeta,{\bf{n}}}\chi_{> 2m}
\end{array}
\right].
\end{align}

Based on the discussion above, their norms are preserved,
\begin{align}
\|\langle v\rangle^{m}(1+K_{\zeta,{\bf{n}}})f\|_{L^1}=\|\langle v\rangle^{m} (1+D)\ \left[
\begin{array}{ll}
\chi_{\leq 2m}f\\
\chi_{>2m}f
\end{array}
\right]\|_{L^1}.\label{eq:decomL1}
\end{align}

Next we prove that all entries in $D$ are small. Recall that the small constants $\delta_{m,l},\ l=0,1,$ are defined in \eqref{eq:difDeltaM0} and \eqref{eq:defDeltaM2}. 
\begin{lemma}\label{LM:estD}
Under the same condition as that in Proposition \ref{prop:roughB}, namely
\begin{align}
|\zeta|\geq (1+|{\bf{n}}|^2)[\Upsilon_m+m]^{4}, \ m\gg 1,\label{eq:largezeta}
\end{align} the four entries of the matrix $D$ satisfy the smallness estimates
\begin{align}
\Big\|\chi_{\leq 2m} \langle v\rangle^{m}K_{\zeta,{\bf{n}}}\langle v\rangle^{-m}\chi_{\leq 2m}\Big\|_{L^1\rightarrow L^1},\ \Big\|\chi_{> 2m}\langle v\rangle^{m} K_{\zeta,{\bf{n}}}\langle v\rangle^{-m}\chi_{\leq 2m}\Big\|_{L^1\rightarrow L^1}\lesssim |\zeta|^{-\frac{1}{2}},\label{eq:zeta}
\end{align} 
\begin{align}
\Big\|\chi_{> 2m}\langle v\rangle^{m} K_{\zeta,{\bf{n}}}\langle v\rangle^{-m}\chi_{> 2m}\Big\|_{L^1\rightarrow L^1}\lesssim \delta_{m,0}\label{eq:geqgeq}
\end{align} and
\begin{align}
\Big\|\chi_{\leq 2m} \langle v\rangle^{m}K_{\zeta,{\bf{n}}}\langle v\rangle^{-m}\chi_{> 2m}\Big\|_{L^1\rightarrow L^1}\lesssim \delta_{m,0}+2^{-m}\label{eq:leqgeq}
\end{align}
\end{lemma}
The lemma will be proved in subsubsection \ref{subsub:LMD}. 

The fact that $D$ is small obviously implies that $1+D$ is uniformly invertible. This is the desired Proposition \ref{prop:roughB}.

Next we prove Lemma \ref{LM:estD} to complete the proof.

\subsubsection{Proof of Lemma \ref{LM:estD}}\label{subsub:LMD}
\begin{proof}
We start with proving \eqref{eq:zeta}. Instead of proving it directly, we observe that the following estimate
\begin{align}
\|\langle v\rangle^{m} K_{\zeta,{\bf{n}}}\langle v\rangle^{-m}\chi_{\leq 2m}\|_{L^1\rightarrow L^1}\lesssim |\zeta|^{-\frac{1}{2}} \label{eq:zeta2}
\end{align} obviously implies the desire two estimates in \eqref{eq:zeta}.

To prove \eqref{eq:zeta2},
we recall that by definition
\begin{align*}
K_{\zeta,{\bf{n}}}:=(K_{1}-K_2-K_3) \Big(\nu(v)+i{\bf{n}}\cdot v-\zeta\Big)^{-1}.
\end{align*}  
Observe that, if $|\zeta|$ is large as in \eqref{eq:largezeta}, then in the set $|v|\leq 2m,$
\begin{align}
|\nu(v)+i{\bf{n}}\cdot v-\zeta|\geq |\zeta|-|{\bf{n}}||v|-|\nu(v)|\geq \frac{1}{4}|\zeta|^{\frac{3}{4}} \langle v\rangle. 
\end{align} This together with the estimate for $K_{l},\ l=1,2,3,$ in \eqref{eq:mK1} implies the desired result
\begin{align}
\| \langle v\rangle^{m}K_{\zeta,{\bf{n}}}\langle v\rangle^{-m}\chi_{\leq 2m}\|_{L^1\rightarrow L^1}\leq 4 |\zeta|^{-\frac{3}{4}}\sum_{l=1}^{3}\| \langle v\rangle^{m}K_{l}\langle v\rangle^{-m-1}\chi_{\leq 2m}\|_{L^1\rightarrow L^1}\leq 4\Upsilon_m |\zeta|^{-\frac{3}{4}}\leq 4|\zeta|^{-\frac{1}{2}}.
\end{align}

It is easy to prove \eqref{eq:geqgeq}. By $|\nu(v)+i{\bf{n}}\cdot v+\zeta|^{-1}\lesssim \langle v\rangle^{-1}$ from \eqref{eq:ReasonCurve}, and recall $\delta_{m,0}$ in Lemma \ref{LM:Control}:
\begin{align}
\|\langle v\rangle^{m}\chi_{> 2m} K_{\zeta,{\bf{n}}}\langle v\rangle^{-m}\chi_{> 2m}\|_{L^1\rightarrow L^1}\lesssim &\sum_{l=1}^{3}\|\langle v\rangle^{m}\chi_{> 2m} K_{l} \langle v\rangle^{-m-1}\chi_{> 2m}\|_{L^1\rightarrow L^1}\leq \delta_{m,0}.\label{eq:grle}
\end{align} 

Next we prove \eqref{eq:leqgeq}. Compute directly to obtain, for any function $f,$
\begin{align}
\|\langle v\rangle^{m}\chi_{\leq 2m} K_{\zeta,{\bf{n}}}\langle v\rangle^{-m}\chi_{> 2m}f\|_{L^1}
\lesssim \sum_{l=1}^{3}\Big\|\langle v\rangle^{m} K_{l} \langle v\rangle^{-m-1}\chi_{> 2m}|f|\Big\|_{L^1}.
\end{align}
Insert the identity $1=\chi_{\leq m}+\chi_{>m}$ before $K_{l}$ to find the desired estimate
\begin{align}
\begin{split}
&\|\langle v\rangle^{m}\chi_{\leq 2m} K_{\zeta,{\bf{n}}}\langle v\rangle^{-m}\chi_{> 2m}f\|_{L^1}\\
\leq &\sum_{l=1}^{3}\Big\|\langle v\rangle^{m}\chi_{\leq m} K_{l} \langle v\rangle^{-m-1}\chi_{> 2m}|f|\Big\|_{L^1}+\sum_{l=1}^{3}\Big\|\langle v\rangle^{m}\chi_{> m} K_{l} \langle v\rangle^{-m-1}\chi_{> 2m}|f|\Big\|_{L^1}\\
\lesssim & 2^{-m}\sum_{l=1}^{3}\Big\| K_{l} \langle v\rangle^{-1}\chi_{> 2m}|f|\Big\|_{L^1}+ \delta_{m,0}\|\chi_{> 2m}f\|_{L^1}\\
\lesssim & [2^{-m}+\delta_{m,0}]\|\chi_{> 2m}f\|_{L^1}
\end{split}
\end{align} where in the last step we used the obvious estimate $\frac{\langle v\rangle^m}{\langle u\rangle^m}\lesssim 2^{-m}$ if $|u|\geq 2m$ and $|v|\leq m,$ and we bound the second term by $\leq \delta_{m,0}\|\chi_{> 2m}f\|_{L^1}$ using Lemma \ref{LM:Control}.

\end{proof}
\subsection{Proof of Proposition \ref{prop:largeN}}\label{sub:LargeN}
We start with presenting the ideas. Recall that $K_{\zeta,\ \bf{n}}$ is defined as
$$K_{\zeta,\ \bf{n}}=K (\nu+i\bf{n}\cdot v-\zeta)^{-1}.$$ We will exploit that if $|\bf{n}|$ is large, then the purely imaginary part of $\nu+i{\bf{n}}\cdot v-\zeta$ is favorably large, 
except for a ``small" set of $v$, namely when $|v|$ is small, or when $\bf{n}$ is almost orthogonal to $v.$

To separate these sets from the others, we define a Heaviside function $\chi$ as
\begin{align}
\chi(v)=\left\{
\begin{array}{ll}
1\ \text{if}\ |v|\leq |\bf{n}|^{-\frac{1}{4}},\\
1\ \text{if}\ \frac{1}{|v||{\bf{n}}|}|{\bf{n}}\cdot v-\zeta_2|\leq |{\bf{n}}|^{-\frac{1}{4}},\\
0\ \text{otherwise}.
\end{array}
\right.
\end{align} Here $\zeta_2$ is defined as the purely imaginary part of $\zeta$, i.e.
\begin{align}
\zeta_2:=Im \ \zeta.
\end{align}

Then as in \eqref{eq:decomL1}
we transform the linear operator $K_{\zeta,{\bf{n}}}$ into a $2\times 2$ operator valued matrix $F$, defined as
\begin{align}
F:=\left[
\begin{array}{cc}
\chi K_{\zeta,{\bf{n}}}\     \chi &\chi K_{\zeta,{\bf{n}}}(1-\chi)\\
(1-\chi) K_{\zeta,{\bf{n}}}\ \chi&(1-\chi) K_{\zeta,{\bf{n}}}(1-\chi)
\end{array}
\right],
\end{align}
and for any function $g\in L^1$,
\begin{align}
\|\langle v\rangle^{m}(1+K_{\zeta,{\bf{n}}}) g\|_{L^{1}}=\|\langle v\rangle^{m}(1+F)\ \left[
\begin{array}{cc}
\chi g\\
(1-\chi)g
\end{array}
\right]\|_{L^1}
\end{align}

Consequently, to prove the invertibility of $1+K_{\zeta,{\bf{n}}}$, it suffices to prove that for the matrix operator $1+F$.

The entries in $F$ satisfy the following estimates: Recall that the small constants $\delta_{m,l},\ l=0,1,$ are defined in \eqref{eq:difDeltaM0} and \eqref{eq:defDeltaM2}. 
\begin{lemma}\label{LM:K123}
There exists $N$ such that if $|{\bf{n}}|\geq N$, then three entries of $F$ are small
\begin{align}
\|\langle v\rangle^{m}\chi K_{\zeta,\bf{n}}\chi f\|_{L^1}\lesssim \Big[ C(m) |{\bf{n}}|^{-\frac{1}{10}}+\delta_{m,0}+\delta_{m,1} \Big] \|\langle v\rangle^{m}\chi f\|_{L^1}, \label{eq:chi11}
\end{align} with $C(m)$ being some constant depending only on $m$,
and
\begin{align}
\|\langle v\rangle^{m}\chi K_{\zeta,\bf{n}} (1-\chi) f\|_{L^{1}},\ \|\langle v\rangle^{m}(1-\chi) K_{\zeta,\bf{n}} (1-\chi) f\|_{L^{1}}\lesssim \Upsilon_m |{\bf{n}}|^{-\frac{1}{4}} \|\langle v\rangle^{m}(1-\chi) f\|_{L^{1}}.\label{eq:estK3}
\end{align}
One (and only one) of the off-diagonal entries is possibly large,
\begin{align}
\|\langle v\rangle^{m}(1-\chi) K_{\zeta,{\bf{n}}}\ \chi f\|_{L^1}\lesssim  \Upsilon_{m}\|\langle v\rangle^{m}\chi f\|_{L^1}.\label{eq:KLarge}
\end{align}

\end{lemma}
The lemma will be proved in subsubsection \ref{subsub:LMK123}. 

We are ready to prove Proposition \ref{prop:largeN}.
\begin{proof}
The difficulty is caused by that an off-diagonal entry, namely $(1-\chi) K_{\zeta,{\bf{n}}} \chi$, is possibly large.

To prove the invertibility of $1+F$ we seek ideas in inverting a $2\times 2$ scalar matrix $\text{Id}+\tilde{F}$: suppose that $\tilde{F}$ takes the form $\tilde{F}=\left[
\begin{array}{ll}
f_{11} & f_{12}\\
f_{21} & f_{22}
\end{array}
\right]$ with $|f_{11}|,\ |f_{12}|, \ |f_{22}|\ll 1$ and $|f_{21}|\gg 1$. But if one has that $|f_{12}f_{21}|\ll 1$, then $\text{Id}+\tilde{F}$ is still invertible and by direct computation,
\begin{align}
\|(\text{Id}+\tilde{F})^{-1}\|\lesssim |f_{21}|.\label{eq:bdInv}
\end{align}

In the present situation, as shown in Lemma \ref{LM:K123}, three entries are small, and only the entry $f_{21}$ is large, but it satisfies the estimate 
\begin{align}
\|f_{12}\|_{L^1\rightarrow L^1}\|f_{21}\|_{L^1\rightarrow L^1}\ll 1.
\end{align}
By this we construct the inverse of $1+F$ by first diagonalizing the matrix, and then finding the bound on the inverse as in \eqref{eq:bdInv}. The process is easy but tedious. We omit the details here.

\end{proof}

\subsubsection{Proof of Lemma \ref{LM:K123}}\label{subsub:LMK123}
\begin{proof}

To prove \eqref{eq:estK3}, the definition of $K$ in \eqref{eq:difK2} makes
\begin{align}
K_{\zeta,{\bf{n}}}f= (K_1-K_2-K_3)(\nu+i{\bf{n}}\cdot v-\zeta)^{-1}f.
\end{align} 
To simplify the treatment, it suffices to prove a slightly more general estimate, for $l=1,2,3,$
\begin{align}
\| \langle v\rangle^{m } K_{l}(\nu+i{\bf{n}}\cdot v-\zeta)^{-1}\langle v\rangle^{-m } (1-\chi) f\|_{L^{1}}\lesssim \Upsilon_m |{\bf{n}}|^{-\frac{1}{4}} \|(1-\chi) f\|_{L^{1}}.
\end{align}

The key observation is that on the support of $1-\chi$, which is the region
\begin{align}
|v|\geq |{\bf{n}}|^{-\frac{1}{4}}\ \text{and}\ \frac{1}{|v||{\bf{n}}|}|{\bf{n}}\cdot v-\zeta_2|\geq |{\bf{n}}|^{-\frac{1}{2}},
\end{align} 
we have
\begin{align*}
|\nu+i{\bf{n}}\cdot v-\zeta |\geq |{\bf{n}}|^{\frac{1}{2}}|v|\geq |{\bf{n}}|^{\frac{1}{4}}(|v|^2+1)^{\frac{1}{2}}.
\end{align*} Here $\zeta_2$ is the purely imaginary part of $\zeta.$
Recall the constant $\Upsilon_m$ from \eqref{eq:mK1}.
Compute directly to obtain the desired result
\begin{align*}
\| \langle v\rangle^{m } K_{l}(\nu+i{\bf{n}}\cdot v-\zeta)^{-1} \langle v\rangle^{-m} (1-\chi) f\|_{L^{1}}\leq& |{\bf{n}}|^{-\frac{1}{4}}\Big\| \langle v\rangle^{m}K_{l}\langle v\rangle^{-1-m} (1-\chi) |f|\Big\|_{L^{1}}\\
\lesssim &\Upsilon_m |{\bf{n}}|^{-\frac{1}{4}}\|(1-\chi) f\|_{L^{1}}.
\end{align*}

\eqref{eq:KLarge} is a simply application of \eqref{eq:mK1}, \eqref{eq:difKsn} and \eqref{eq:ReasonCurve},
\begin{align}
\|\langle v\rangle^{m}(1-\chi) K_{\zeta,{\bf{n}}}\ \chi f\|_{L^1}\lesssim  \sum_{l=1}^{3} \Big\|\langle v\rangle^{m} K_{l}\langle v\rangle^{-1}\ \chi |f|\Big\|_{L^1}\lesssim\Upsilon_{m}\|\langle v\rangle^{m}\chi f\|_{L^1}.
\end{align}

Next we prove \eqref{eq:chi11}.

Decompose the operator $\chi K_{l} (\nu+i{\bf{n}}\cdot v-\zeta)^{-1}\chi,\ l=1,2,3,$ further, by inserting the identities $1=\chi_{\leq 2m}+\chi_{>2m}$ and $1=\chi_{\leq m}+\chi_{>m}$, into appropriate places, to find
\begin{align}
\chi K_{l}(\nu+i{\bf{n}}\cdot v-\zeta)^{-1}\chi
=&\chi_{\leq 2m} \chi K_{l}(\nu+i{\bf{n}}\cdot v-\zeta)^{-1}\chi+\chi_{>2m} \chi K_{l}(\nu+i{\bf{n}}\cdot v-\zeta)^{-1}\chi_{\leq m}\chi\nonumber\\
&+\chi_{>2m} \chi K_{l}(\nu+i{\bf{n}}\cdot v-\zeta)^{-1}\chi_{> m}\chi\label{eq:decomThree}
\end{align}
where $\chi_{\leq 2m}$, $\chi_{>2m}=1-\chi_{\leq 2m},$ $\chi_{>m},$ and $\chi_{\leq m}=1-\chi_{>m}$ are Heaviside functions defined in a similar way as that in \eqref{eq:cutoff2}.

For the first term on the right hand side, i.e. $\chi_{\leq 2m} \chi K_{l}(\nu+i{\bf{n}}\cdot v-\zeta)^{-1}\chi,\ l=1,2,3$, we use the fact that $|(\nu+i{\bf{n}}\cdot v-\zeta)^{-1}|\lesssim \langle v\rangle^{-1}\lesssim 1$ to find that, for any function $f,$
\begin{align}
\|\langle v\rangle^{m}\chi_{\leq 2m} \chi K_{l}(\nu+i{\bf{n}}\cdot v-\zeta)^{-1}\chi \langle v\rangle^{-m} f\|_{L^1}\lesssim  \Big\|\langle v\rangle^{m}\chi_{\leq 2m} \chi K_{l} \chi  \langle v\rangle^{-m}|f|\Big\|_{L^1}.\label{eq:prelimEst}
\end{align}

Now we claim that, there exists some constant $C(m)>0$, such that for  $l=1,2,3,$
\begin{align}
\Big\|\langle v\rangle^{m}\chi_{\leq 2m} \chi K_{l} \chi  \langle v\rangle^{-m}|f|\Big\|_{L^1}\leq \Big\|\langle v\rangle^{m}\chi_{\leq 2m} \chi K_{l}\chi  |f|\Big\|_{L^1}\leq  C(m)|{\bf{n}}|^{-\frac{1}{10}} \|\chi  f\|_{L^1}.\label{eq:finite}
\end{align}
This implies the desired estimate \eqref{eq:chi11} for the considered part.

Now we prove the claim. The firs step is trivial. For the second and for $l=2,3,$ use the integral kernel for $K_2+K_3$ in \eqref{eq:difK} to find
\begin{align}
\sum_{l=2,3}\Big\|\langle v\rangle^{m}\chi_{\leq 2m} \chi K_l  \chi |f| \Big\|_{L^1}
= \Big\|\langle v\rangle^{m}\chi_{\leq 2m} \chi [K_{2}+K_3]\chi  |f|\Big\|_{L^1}
\lesssim \max_{u} \int_{\Omega} \langle v\rangle^{m} |v-u|^{-1} dv^3\ \|\chi f\|_{L^1},\label{eq:finite2}
\end{align} where the set $\Omega=\Omega_1\cup \Omega_2\subset \mathbb{R}^3$ is defined as
\begin{align}
\begin{split}
\Omega_1:=&\Big\{ v\in \mathbb{R}^3\ \Big| \ |v|\leq |{\bf{n}}|^{-\frac{1}{4}}\Big\},\\
\Omega_2:=&\Big\{ v\in \mathbb{R}^3\ \Big| \ |v|\leq 2m,\ \frac{1}{|v||{\bf{n}}|}|{\bf{n}}\cdot v-\zeta_2|\leq |{\bf{n}}|^{-\frac{1}{4}},\ |v|\geq |{\bf{n}}|^{-\frac{1}{4}}\Big\}.
\end{split}
\end{align} Observe that the larger value of $|\bf{n}|$, the smaller the volume of the set $\Omega$ becomes. Next we exploit this. It is easy to control the integration in the region $\Omega_1$ since for any $u\in \mathbb{R}^3$
\begin{align}
\int_{|v|\leq {\bf{n}}^{-\frac{1}{4}}}|v-u|^{-1}\ dv\leq |{\bf{n}}|^{-\frac{1}{4}}.
\end{align}

For $\Omega_2$, compute directly to find that, 
\begin{align}
\max_{u} \int_{\Omega_2} \langle v\rangle^{m} |v-u|^{-1} dv^3 \lesssim &(1+2m)^{m} \Big[\int_{|v-u|\leq |\bf{n}|^{-\frac{1}{10}}} |v-u|^{-1} d^3v+ |{\bf{n}}|^{\frac{1}{10}} \int_{\Omega} d^3v\Big] \nonumber\\
=&(1+2m)^{m} \Big[\int_{|v|\leq |\bf{n}|^{-\frac{1}{10}}} |v|^{-1} d^3v+ |{\bf{n}}|^{\frac{1}{10}} \int_{\Omega} d^3v\Big].
\end{align}
The first term is easy to be controlled, for some constant $C_1(m)>0$,
\begin{align}
(1+2m)^{m} \int_{|v|\leq |\bf{n}|^{-\frac{1}{10}}} |v|^{-1} d^3v\leq C_1(m)  |\bf{n}|^{-\frac{1}{10}}.\label{eq:c1m}
\end{align} 
For the second one, we need to evaluate the integral $\int_{\mathcal{R}} d^3v.$ The key is to show the set $\Omega$ becomes smaller, as $|\bf{n}|$ increases. Without losing generality, assume that ${\bf{n}}=(|{\bf{n}}|,0,0)$. Then the condition $\frac{1}{|v||{\bf{n}}|}|{\bf{n}}\cdot v-\zeta_2|\leq |{\bf{n}}|^{-\frac{1}{4}}$ becomes
\begin{align*}
\Big|(1,0,0)\cdot \frac{v}{|v|}-\frac{\zeta_{2}}{|v||{\bf{n}}|}\Big|\leq |{\bf{n}}|^{-\frac{1}{4}}.
\end{align*}
Thus, for each fixed $|v|$, the set of vectors $\frac{v}{|v|}\in \mathbb{S}^2$ can only be a subset of $\mathbb{S}^2$ with area $\mathcal{O}(|{\bf{n}}|^{-\frac{1}{4}}).$ From here we integrate in polar coordinate to find that, for some constant $C_2(m)$,
\begin{align}
(1+2m)^{m}  |{\bf{n}}|^{\frac{1}{10}} \int_{\Omega} d^3v\leq C_2(m)  |\bf{n}|^{-\frac{1}{10}}.
\end{align}

This, together with \eqref{eq:finite2}-\eqref{eq:c1m}, implies the desired \eqref{eq:finite} for $K_{2}-$ and $K_{3}-$terms.

It is easier to estimate the $K_1-$term, rendered by the factor $e^{-|v|^2}$ in its integral kernel. We skip the details here.

For the third term in \eqref{eq:decomThree}, we apply Lemma \ref{LM:Control} to find the desired estimate
\begin{align}
\begin{split}
\|\chi_{>2m}\langle v\rangle^m \chi K_{l}(\nu+i{\bf{n}}\cdot v-\zeta)^{-1}\langle v\rangle^{-m} \chi\ \chi_{> m}\|_{L^1\rightarrow L^1}\leq &\|\chi_{>2m}\langle v\rangle^m K_{l}\langle v\rangle^{-m-1}  \chi_{> m}\|_{L^1\rightarrow L^1}\\
\leq &\delta_{m,0}.
\end{split}
\end{align}

Turning to the second term in \eqref{eq:decomThree}, we apply \eqref{eq:ReasonCurve2} to find, for any function $f$, $l=1,2,3$
\begin{align}
\Big\|\chi_{>2m}\chi \langle v\rangle^m K_{l}(\nu+i{\bf{n}}\cdot v-\zeta)^{-1}\langle v\rangle^{-m} \chi\ \chi_{\leq m} f\Big\|_{L^1}\leq \Big\|\chi_{>2m} \langle v\rangle^{m}K_{l}  \langle v\rangle^{-m-1}\chi_{\leq m}\chi |f|\Big\|_{L^1}.
\end{align}
Among the three terms,  $K_1$-term is the easiest, the factor $e^{-|v|^2}$ in its integral kernel makes the integral small,
\begin{align}\label{eq:k1}
\begin{split}
 \|\chi_{>2m}\langle v\rangle^{m}  K_{1}\langle v\rangle^{-m-1} \chi_{\leq m}\chi f\|_{L^1}
 \lesssim &\int_{|v|\geq 2m} \langle v\rangle^{m+1} e^{-|v|^2}\ d^3v\ \|\chi_{\leq m} \chi f\|_{L^1}\\
 \leq &\delta_{m,1}\|\chi_{\leq m} \chi f\|_{L^1} 
 \end{split}
\end{align} where the small constant $\delta_{m,1}$ is defined in \eqref{eq:mInte}.

We estimate the $K_2-$ and $K_3-$terms together. Use the integral kernel of $K_2+K_3$ to find
\begin{align}
\begin{split}\label{eq:k2k3}
&\sum_{l=2,3} \Big\|\chi_{>2m}\langle v\rangle^{m}  K_{l}\langle v\rangle^{-m-1} \chi_{\leq m} \chi |f|\Big\|_{L^1}\\
=& \Big\|\chi_{>2m}\langle v\rangle^{m} [ K_{2}+K_3]\langle v\rangle^{-m-1} \chi_{\leq m} \chi |f|\Big\|_{L^1}\\
=&  \int_{|v|\geq 2m,\ |u|\leq m}\langle v\rangle^{m}  K(u,v)\langle u\rangle^{-m-1}  \chi(u) |f|(u)\ d^3u d^3v
\end{split}
\end{align}
where by \eqref{eq:difK}, $K(u,v)$ takes the form $$K(u,v):=2\pi  |u-v|^{-1} e^{ -\frac{|(u-v)\cdot v|^2}{|u-v|^2}}.$$ The facts that $|v|\geq 2m$ and $|u|\leq m$ imply that the direction of the unit vector
$\frac{u-v}{|u-v|}$ is controlled by $v$, a detailed analysis yields
\begin{align*}
\frac{|(u-v)\cdot v|^2}{|u-v|^2}\geq \frac{1}{4}|v|^2.
\end{align*}
This together with estimate $|u-v|^{-1}\leq 2|v|^{-1}$ implies that
\begin{align}
K(u,v)\lesssim |v|^{-1} e^{-\frac{1}{4}|v|^2}.
\end{align}
Put this back into \eqref{eq:k2k3}, and recall the definition of $\delta_{m,1}$ in \eqref{eq:defDeltaM2}, to obtain the desired estimate
\begin{align}
\begin{split}
\sum_{l=2,3} \Big\|\chi_{>2m}\langle v\rangle^{m}  K_{l}\langle v\rangle^{-m-1} \chi_{\leq m}\chi |f|\Big\|_{L^1}\lesssim &\int_{|v|\geq 2m} \langle v\rangle^{m} |v|^{-1} e^{-\frac{1}{4}|v|^2}\ dv^3 \|\chi f\|_{L^1}\\
\leq &\delta_{m,1} \|\chi f\|_{L^1}.
\end{split}
\end{align}

This together with \eqref{eq:k1} completes the estimate for the second term in \eqref{eq:decomThree}.

Since we hav found all the terms in \eqref{eq:decomThree} satisfied the desired estimate, the proof of \eqref{eq:chi11} is complete.
\end{proof}




\subsection{Proof of Proposition \ref{prop:compact}}\label{sec:matchingArgu}
\begin{proof}
In what follows we take some ideas and constructions in \cite{Mouhot2006, Mouhot2010}. However the arguments here are largely self-contained, and are more direct.

In the next we use the construction of (right) inverse operator for $\langle v\rangle^{m}(L_{\bf{n}}-\lambda) \langle v\rangle^{-m}$ in Proposition 4.1 of \cite{Mouhot2006}.
Define a Heaviside function $\Theta_{\delta}=\Theta_{\delta}(v, u)$ as
\begin{align}\label{eq:cutoff20}
\Theta_{\delta}(v,u)= 
\left[
\begin{array}{ll}
1\ &\text{if}\ \sqrt{|v|^2+|u|^2}\leq \delta^{-1},\\
0 &\ \text{otherwise}.
\end{array}\right.
\end{align} Here $\delta$ is a small constant chosen to satisfy 
\begin{align}
\frac{\delta^{-1}}{m^{10}}\gg 1,\ \text{and} \ m\gg 1.
\end{align}
Recall the definitions of the operators $K_{l},\ l=1,2,3$ from \eqref{eq:difK2}, we define $K_{\delta}$ by inserting the cutoff function
\begin{align}
K_{\delta}(f)
:=& M(v) \int_{\mathbb{R}^{3}\times \mathbb{S}^2}\Theta_{\delta}(v, u)\ |(u-v)\cdot\omega|f(u)\ d^3 ud^2 \omega\nonumber\\
& - \int_{\mathbb{R}^3\times\mathbb{S}^2 }\Theta_{\delta}(v, u) \ |(u-v)\cdot\omega |M(u') f(v')\ d^3 ud^2 \omega\nonumber\\
& -\int_{\mathbb{R}^3 \times \mathbb{S}^2}\Theta_{\delta}(v, u)\ |(u-v)\cdot\omega| M(v') f(u')\ d^3 ud^2 \omega\nonumber.
\end{align} Here recall that we only consider the case $T=\frac{1}{2},\ \mu=0$, and denote $M=M_{\frac{1}{2},0},$ see \eqref{eq:abbre}.

Decompose the operator $\langle v\rangle^{m}(L_{\bf{n}}-\lambda) \langle v\rangle^{-m}$ into two parts,
$$\langle v\rangle^{m}(L_{\bf{n}}-\lambda) \langle v\rangle^{-m}=B_{\delta}+A_{\delta}$$ with $B_{\delta}:=\nu+i{\bf{n}}\cdot v -\lambda+\langle v\rangle^{m}[K-K_\delta]\langle v\rangle^{-m}$, and $A_{\delta}:=\langle v\rangle^{m}K_{\delta}\langle v\rangle^{-m}.$
The construction of the inverse of $\langle v\rangle^{m}(L_{\bf{n}}-\lambda)\langle v\rangle^{-m}:\ L^1\rightarrow L^1$ is taken from \cite{Mouhot2006},
\begin{align}
&[\langle v\rangle^{m}(L_{\bf{n}}-\lambda)\langle v\rangle^{-m}] \times\nonumber\\
& \Big[1-\langle v\rangle^{m}M^{\frac{1}{2}} [M^{-\frac{1}{2}}(L_{\bf{n}}-\lambda)M^{\frac{1}{2}}]^{-1} \langle v\rangle^{-m} M^{-\frac{1}{2}}\ \big[\langle v\rangle^{m}A_{\delta}\langle v\rangle^{-m}\big]\Big] [\langle v\rangle^{m} B_{\delta}\langle v\rangle^{-m}]^{-1}=Id.
\end{align}

Hence if the operator $$\Big[1-\langle v\rangle^{m}M^{\frac{1}{2}} \big[M^{-\frac{1}{2}}(L_{\bf{n}}-\lambda)M^{\frac{1}{2}}\big]^{-1} \langle v\rangle^{-m} M^{-\frac{1}{2}}\ \big[\langle v\rangle^{m}A_{\delta}\langle v\rangle^{-m}\big] 
\Big] \ \Big[\langle v\rangle^{m} B_{\delta}\langle v\rangle^{-m}\Big]^{-1}: L^1(\mathbb{R}^3)\rightarrow L^{1}(\mathbb{R}^3)$$ is well defined, then $\langle v\rangle^{m}(L_{\bf{n}}-\lambda)\langle v\rangle^{-m}$ is invertible, which directly implies the desired result.
To verify the operator is well defined, we use the following facts:
\begin{itemize}
\item[(1)] the mapping $\langle v\rangle^{m}M^{\frac{1}{2}}:\ L^2\rightarrow L^1$ is obviously well defined since $M$ decays rapidly fast, 
\item[(2)]the mapping $[M^{-\frac{1}{2}}(L_{\bf{n}}-\lambda)M^{\frac{1}{2}}]^{-1}:\ L^2\rightarrow L^2$ is well defined since $-\lambda$ is not an eigenvalue of $L_{\bf{n}},$ see \cite{YGuo2002, YGuo2003}. Recall that ${\bf{n}}\not=(0,0,0),$
and recall that, in certain sense, the eigenvector space of $L_{\bf{n}}$ is a subset of that of $L$, see Lemma \ref{LM:subspace}.
\item[(3)] the mapping $M^{-\frac{1}{2}}A_{\delta}\langle v\rangle^{-m}$ can be considered in the setting 
\begin{align}\label{eq:setting}
M^{-\frac{1}{2}}A_{\delta}\langle v\rangle^{-m}: \ L^1\rightarrow L^2
\end{align} by using that $K_{\delta}h$ is ``compactly supported" by the definition of cutoff function \eqref{eq:cutoff20}. 
\item[(4)]
To show that $\langle v\rangle^{m} B_{\delta}\langle v\rangle^{-m}: L^1\rightarrow L^1$ is invertible for large $m$, we use the identity
\begin{align}
 \langle v\rangle^{m} B_{\delta}\langle v\rangle^{-m}= \Big(1+\langle v\rangle^{m}\Big(K-K_{\delta}\Big)(\nu+iv\cdot {\bf{n}}-\lambda)^{-1}\langle v\rangle^{-m}\Big)\Big(\nu+iv\cdot {\bf{n}}-\lambda\Big).\label{eq:twofactors}
\end{align} It is easy to see that the second factor is invertible by the estimate 
\begin{align}
|(\nu(v)+iv\cdot {\bf{n}}-\lambda)^{-1}|\leq C (1+|v|)^{-1}, \label{eq:uppper}
\end{align} which is from \eqref{eq:ReasonCurve2}.

For the first factor we show that $\langle v\rangle^{m}\Big(K-K_{\delta}\Big)(\nu+iv\cdot {\bf{n}}-\lambda)^{-1}\langle v\rangle^{-m}: \ L^1\rightarrow L^1$ is small by the following result:
\begin{lemma}\label{LM:firFactor}
 Suppose that $m$ is sufficiently large and $ m^{-10}\delta^{-1}\gg 1.$
 Then we have that for any function $f$,
 \begin{align}
  \|\langle v\rangle^{m} \Big(K-K_{\delta}\Big) \Big(\nu+i\cdot {\bf{n}}-\lambda\Big)^{-1} \langle v\rangle^{-m} f\|_{L^1}\leq \frac{1}{4}\|f\|_{L^1}.
 \end{align}

\end{lemma}
The lemma will be proved in Subsubsection \ref{subsec:LMFF} below.

Suppose that the lemma holds, then the first factor in \eqref{eq:twofactors} is invertible. This implies that $\langle v\rangle^{m} B_{\delta}\langle v\rangle^{-m}: L^1\rightarrow L^1$ is invertible.


\end{itemize}

The proof is complete.
\end{proof}

\subsubsection{Proof of Lemma \ref{LM:firFactor}}\label{subsec:LMFF}
We choose to study one term in $K-K_{\delta}$, which is
\begin{align}\label{eq:defD2}
D_2 f:=\langle v\rangle^{m}\int_{\mathbb{R}^3\times\mathbb{S}^2 }[1-\Theta_{\delta}(v, u)] \ |(u-v)\cdot\omega |M(u') \Big(\nu(v')+iv'\cdot {\bf{n}}-\lambda\Big)^{-1} \langle v'\rangle^{-m} f(v')\ d^3 ud^2 \omega.
\end{align}
Its $L^1$ norm satisfies the estimate, using \eqref{eq:uppper},
\begin{align}
\|D_2 f\|_{L^1}\lesssim \int_{\mathbb{S}^2}\int_{|u|^2+|v|^2\geq \delta^{-2} }\langle v\rangle^{m} |(u-v)\cdot \omega| e^{-|u'|^2} \langle v' \rangle^{-m-1} |f|(v') \ d^3ud^3v d^2\omega.
\end{align}

Now we change variables. Recall that $u', v'\in \mathbb{R}^{3}$ are given by $$u':=u-[(u-v)\cdot \omega]\omega,\ v':=v+[(u-v)\cdot\omega]\omega.$$ This makes
\begin{align}
u:=u'-[(u'-v')\cdot \omega] \omega,\ v=v'+[(u'-v')\cdot \omega] \omega.
\end{align} and $|u|^2+|v|^2=|u'|^2+|v'|^2$, and $|(u-v)\cdot\omega|=|(u'-v')\cdot \omega|$, and hence change variables $(u,v)\rightarrow (u',v')$ to find,
\begin{align}
\begin{split}
\|D_2 f\|_{L^1}\lesssim &\int_{|u'|^2+|v'|^2\geq \delta^{-2} }\langle v\rangle^{m} |(u'-v')\cdot \omega| e^{-|u'|^2} \langle v' \rangle^{-m-1} |f|(v') \ d^3u'd^3v' d^2\omega\\
=&\int_{|u|^2+|v|^2\geq \delta^{-2} }\langle v'\rangle^{m} |(u-v)\cdot \omega| e^{-|u|^2} \langle v \rangle^{-m-1} |f|(v) \ d^3ud^3v d^2\omega.
\end{split}
\end{align}
Let $U_{\omega}$ be the rotation in \eqref{eq:rota}, which makes $U_{\omega}^* \omega=(1,0,0)^{T}$. Then we rotate both $u$ and $v$ to have
\begin{align}\label{eq:manStep}
\|D_2 f\|_{L^1}\lesssim \int_{|u|^2+|v|^2\geq \delta^{-2} }(1+u_1^2+v_2^2+v_3^2)^{\frac{m}{2}} |u_1-v_1| e^{-|u|^2} \langle v \rangle^{-m-1} \Big[ \int_{\mathbb{S}^2}|f(U_{\omega}v)|\ d^2\omega\Big] d^3u d^3 v .
\end{align}

From here we adopt the same strategy as estimating the term in \eqref{eq:smallIn2}, and the present problem is easier since we have the condition $m^{-10}\delta^{-1}\gg 1$. In the polar coordinate, $v$ is of the form 
\begin{align}
v_1=r\cos\alpha,\ v_2=r\sin\alpha \cos\beta,\ v_3=r\sin\alpha \sin\beta,
\end{align} with $r\geq 0,$ $\alpha\in [0,\pi]$ and $\beta\in [0,\ 2\pi].$
Compute directly to find
\begin{align}
\|D_2 f\|_{L^1} \lesssim \|f\|_{L^1}\tilde{D}_2\label{eq:D2f}
\end{align}
 with $\tilde{D}_2$ defined as
\begin{align}\label{eq:tildD}
\tilde{D}_2=\sup_{r\geq 0,\ \alpha}\Big[(1+r^{2})^{-\frac{m+1}{2}}\int_{|u|^2+r^2\geq \delta^{-2} } (1+u_1^2+r^2\sin^2\alpha)^{\frac{m}{2}} \Big[|u_1|+r|\cos\alpha|\Big] e^{-|u|^2}\ d^3 u \Big].
\end{align}

Next we discuss three regions separately, specifically $$\Big\{ u\ \Big|\ |u|\geq \delta^{-\frac{1}{4}}\Big\}, \ \Big\{(u,v)\ \Big|\ |u|< \delta^{-\frac{1}{4}},\ r\sin\alpha\leq \delta^{-\frac{3}{4}}\Big\},\ \text{and}\  
\Big\{(u,v)\ \Big|\ |u|< \delta^{-\frac{1}{4}},\ r\sin\alpha> \delta^{-\frac{3}{4}}\Big\}.$$

The first and second ones are easy. By the rapid decay of $e^{-|u|^2}$, it is easy to see that in the region $|u|\geq \delta^{-\frac{1}{4}}$, the integral is of order $\delta^{10}.$
For the second, we observe that
\begin{align}
\frac{1+u_1^2+r^2\sin^2\alpha}{1+r^2}\leq \delta^{\frac{1}{2}},
\end{align} from here
compute directly to find that the integral is of order $\delta^{\frac{1}{2}}.$

For the third one, we use $m^{-10}\delta^{-1}\gg 1$ to find
\begin{align}
(1+u_1^2+r^2\sin^2\alpha)^{\frac{m}{2}}(|u_1|+r|\cos\alpha|)\leq 5 r^{m+1}\sin^{m}\alpha (\delta^{\frac{1}{4}}+|\cos\alpha|). 
\end{align}
Put this back into \eqref{eq:tildD} and use that
\begin{align}
\gamma(m):=\sup_{\alpha}\big|\sin^{m}\alpha\ \cos\alpha\big|\rightarrow 0,\ \text{as}\ m\rightarrow \infty
\end{align} to find that the integral is small.

Collect the estimates above to have that
\begin{align}
\tilde{D}_2\lesssim  \gamma(m)+\delta^{\frac{1}{4}}.
\end{align}

This together with \eqref{eq:D2f} implies the desired estimate for $D_2$ of \eqref{eq:defD2}.

The other parts in $K-K_{\delta}$ can be controlled similarly, together with the techniques used in proving Lemma \ref{LM:Control}, hence we skip the details here.


\appendix
\section{Proof of Lemma \ref{LM:EstNonline}}\label{sec:nonlinear}
To derive \eqref{eq:lowB} and \eqref{eq:globalLower}, we recall the definitions of $\nu_{T,\mu}$ in \eqref{eq:difNu}
\begin{align}
\nu_{T,\mu}(v):=\int_{\mathbb{R}^3\times \mathbb{S}^2}|(u-v)\cdot\omega| M_{T,\mu}(u)\ d^3u d^2 \omega.
\end{align} The $\omega-$integral has a closed form, 
\begin{align}
 \nu_{T,\mu}(v)=4\pi \int_{\mathbb{R}^3}|u-v|M_{T,\mu}(u)\ d^3u.
\end{align} 
From here we compute directly to obtain, for some $C>0,$
\begin{align}
  \nu_{T,\mu}(v)\geq C (1+|v|)
\end{align} which implies the desired \eqref{eq:lowB} and \eqref{eq:globalLower}.
For a more general consideration, see \cite{BoGamPan04}. 

For \eqref{eq:estNonL}, we start with proving that, for any functions $f,\ g:\ \mathbb{R}^3\rightarrow \mathbb{C}$ we have
\begin{align}
\|\langle v\rangle^{m} Q(f,g)\|_{L^{1}(\mathbb{R}^{3})}
\leq &C_{m} [\|f\|_{L^{1}(\mathbb{R}^{3})}\|\langle v\rangle^{m+1}g\|_{L^{1}(\mathbb{R}^{3})}+ \|\langle v\rangle^{m+1}f\|_{L^{1}(\mathbb{R}^{3})}\|g\|_{L^{1}(\mathbb{R}^{3})}].\label{eq:EForm}
\end{align}

A key observation in proving the estimate is that for any fixed $\omega\in \mathbb{S}^{2},$ the mapping from $(u,v)\in \mathbb{R}^{6}$ to $(u',v')\in \mathbb{R}^{6}$ is a linear symplectic transformation, hence
\begin{equation}
d^3u d^3 v=d^3 u' d^3 v'
\end{equation} where, $u'$ and $v'$ are defined \eqref{eq:NLBL1}. This together with the observation that
\begin{equation}
\langle v\rangle^{m}\leq c(m)\big[ \langle u'\rangle^{m}+\langle v'\rangle^{m}\big],\ \text{and}\ |(u-v)\cdot \omega|\leq |u'|+|v'|
\end{equation} obviously implies \eqref{eq:EForm}, and hence the desired \eqref{eq:estNonL}.

As one can infer from the definition $K$ in \eqref{eq:difK2}, \eqref{eq:mK1} is a special case of \eqref{eq:estNonL} by setting $f$ or $g$ to be $M_{T,\mu}.$
\begin{flushright}
$\square$
\end{flushright}



\section{The local wellposedness of the linear equation}\label{sec:Local}

Here we study the local wellposedness, in weighted $L^{1}-$space, of the linear problem
\begin{align}
\begin{split}\label{local}
\partial_{t}g=&[-\nu-i{\bf{n}}\cdot v-K]g,\\
g(v,0)=&g_0(v).
\end{split}
\end{align}

Recall that we study the solution in the space $$\langle v\rangle^{-m}L^1(\mathbb{R}^3):=\Big\{f\ \Big|\ \|\langle v\rangle^{m}f\|_{L^1}<\infty\Big\},$$ with $m$ sufficiently large. 

The main result is:
\begin{lemma}\label{LM:wellpose}
Suppose that $m>0$ is large enough, then the equation \eqref{local} has a unique solution for any given $g_{0}\in \langle v\rangle^{-m}L^1(\mathbb{R}^3)$. Moreover, for any $t\geq 0,$ there exists a positive function $X$, independent of $\bf{n}$, such that
\begin{align}
\|\langle v\rangle^{m}g(\cdot,t)\|_{L^1}\leq X(t) \|\langle v\rangle^{m}g_0\|_{L^1}.
\end{align}
\end{lemma}
\begin{proof}
We start with casting the equation into a convenient form by
applying Duhamel's principle
\begin{align}
g=e^{(-\nu-i{\bf{n}}\cdot v)t}g_0-\int_{0}^{t} e^{(-\nu-i{\bf{n}}\cdot v)(t-s)} K g(s)\ ds.\label{eq:durh}
\end{align}

We start with simplifying the problem.
\begin{itemize}
\item[(1)]
Since the equation is linear, it suffices to prove the existence of solutions in a small time interval. 
\item[(2)]
All the estimates made on the terms on the right hand side of \eqref{eq:durh} will be based on \eqref{eq:fix1} and \eqref{eq:fix2} below, which do not depend on $\bf{n}$. Thus the estimates are ``uniform in $\bf{n}$".
\end{itemize}

The main tool will be the fixed point theorem. To make it applicable we define a Banach space.

We define the norm, for any function $g:\mathbb{R}^3\times \mathbb{R}^{+}\rightarrow \mathbb{C}$, for any $\tau\geq 0,$
\begin{align}
\|g\|_{\tau}:=\max_{s\in [0,\tau]}\Big[\|\langle v\rangle^{m} g(s)\|_{L^1}+\Phi^2\int_{0}^{s}\|\chi_{\leq m}\langle v\rangle^{m+1}g(s_1)\|_{L^1}\ ds_1+\Phi\int_{0}^{s}\|\chi_{> m}\langle v\rangle^{m+1}g(s_1)\|_{L^1}\ ds_1\Big].\label{eq:normWell}
\end{align}
where $\Phi$ is large constant to be chosen later. The ideas in choosing the norm above are motivated directly by those used in \cite{YGuo2003, YGuo2002}, for a different approach see the application of Lumer-Philipps Theorem in \cite{Mouhot2010}.

In the chosen Banach space, the following two results make the fixed point theorem applicable, hence establish the desired result Lemma \ref{LM:wellpose}:
\begin{itemize}
\item[(A)]
for any $\tau>0$, and $g_0$ satisfying $\|\langle v\rangle^{m}g_0\|_{L^1}< \infty$, we have
\begin{align}
\|e^{-\nu t}|g_0|\|_{\tau}<\infty,\label{eq:fix1}
\end{align}
\item[(B)]
if $\tau>0$ is sufficiently small, then the linear mapping $\int_{0}^{t} e^{(-\nu-i{\bf{n}}\cdot v)(t-s)} K g(s)\ ds$ is contractive,
\begin{align}
\|\int_{0}^{t} e^{(-\nu-i{\bf{n}}\cdot v)(t-s)} K g(s)\ ds\|_{\tau}\leq \frac{3}{4}\|g\|_{\tau}.\label{eq:fix2}
\end{align} 
\end{itemize}

To complete the proof, we need to prove the two key estimates \eqref{eq:fix1} and \eqref{eq:fix2}.

We start with proving \eqref{eq:fix2}. To simplify the notation, we define a linear operator $H(g)$ by
\begin{align}
H (g)(t):=\sum_{l=1}^{3}\int_{0}^{t} e^{-\nu(t-s)} K_{l} g(s)\ ds.
\end{align}

We start with estimating $\|\langle v\rangle^{m}H(g)\|_{L^1}$. Compute directly to obtain, recall $\Upsilon_m=\Upsilon_{m,\frac{1}{2}}$ from \eqref{eq:mK1},
\begin{align}
\|\langle v\rangle^{m}H (g)\|_{L^1}\lesssim &\Upsilon_m\int_{0}^{t} \|\langle v\rangle^{m+1}g(s)\|_{L^1}\ ds\nonumber\\
=&\Upsilon_{m}\int_{0}^{t} \|\chi_{\leq m}\langle v\rangle^{m+1}g(s)\|_{L^1}\ ds+\Upsilon_m\int_{0}^{t} \|\chi_{> m}\langle v\rangle^{m+1}g(s)\|_{L^1}\ ds
\leq\frac{\Upsilon_m}{\Phi} \|v\|_{t}.\label{eq:contr1}
\end{align}

For $\Phi^2\int_{0}^{t}\|\chi_{\leq m}\langle v\rangle^{m+1}H(g)\|_{L^1}\ ds$, compute directly to obtain
\begin{align}
\Phi^2\int_{0}^{t}\|\chi_{\leq m}\langle v\rangle^{m+1}H(g)\|_{L^1}\ ds
 \leq & \Phi^2\int_{0}^{t}\int_{0}^{s} \|\chi_{\leq m}e^{-(s-s_1)\nu}\langle v\rangle^{m+1}K\ g(s_1)\|_{L^1}\ ds_1 ds\nonumber\\
\leq & \Phi^2\sum_{k=1}^{3} \|\chi_{\leq m}\int_{0}^{t}\int_{0}^{s} e^{-(s-s_1)\nu}\langle v\rangle^{m+1}K_{l}\ |g(s_1)|\ ds_1 ds\|_{L^1}.\label{eq:intNorm}
\end{align}

Integrate by parts in $s$, using $e^{-s\nu}=-\nu^{-1} \partial_{s} e^{-s\nu}$, to obtain
\begin{align}
\chi_{\leq m}\int_{0}^{t}\int_{0}^{s} e^{-(s-s_1)\nu}\langle v\rangle^{m+1}K_{l}\ |g(s_1)|\ ds_1 ds&= \chi_{\leq m}\int_{0}^{t}(1- e^{-(t-s)\nu}) \nu^{-1} \langle v\rangle^{m+1}K_{l}\ |g(s)|\ ds.\label{eq:InteBP}
\end{align} Here we exploit that $\chi_{\leq m}(1- e^{-(t-s)\nu})$ is small, provided that $t$, and hence $s\leq t$, are sufficiently small. To see this,  define $\epsilon(tm)$ as
\begin{align}
 \epsilon(t m) :=\|\chi_{\leq m} (1- e^{-t\nu})\|_{L^\infty}\leq \|\chi_{\leq m} (1- e^{-(t-s)\nu})\|_{L^\infty},\ s\leq t.
\end{align}
It is easy to see that 
\begin{align}
\epsilon(t m)\rightarrow 0\ \text{as}\ t m\rightarrow 0.
\end{align}

Plug this into \eqref{eq:intNorm}, and use $\nu^{-1} \langle v\rangle\lesssim 1,$ to obtain
\begin{align}
&\Phi^2\int_{0}^{t}\|\chi_{\leq m}\langle v\rangle^{m+1}H(g)\|_{L^1}\ ds\nonumber\\
\lesssim & \Phi^2\epsilon(tm) \Upsilon_m \big[\int_{0}^{t}\|\chi_{\leq m}\langle v\rangle^{m+1}g(s)\|_{L^1}\ ds
+\int_{0}^{t}\|\chi_{> m}\langle v\rangle^{m+1}g(s)\|_{L^1}\ ds\big]\nonumber\\
\lesssim & (\Phi+1)\epsilon(tm) \Upsilon_m\|g\|_{t}. \label{eq:contr2}
\end{align}

For
$\Phi\int_{0}^{t}\|\chi_{> m}\langle v\rangle^{m+1}H(g)\|_{L^1}\ ds$, integrate by parts as in \eqref{eq:InteBP} to find
\begin{align}
\int_{0}^{t}\|\chi_{> m}\langle v\rangle^{m+1}H(g)\|_{L^1}\ ds\lesssim &\sum_{l=1}^3\big[ \int_{0}^{t}\|\chi_{> m}\langle v\rangle^{m}K_l \chi_{>m}g(s)\|_{L^1}\ ds+\int_{0}^{t}\|\chi_{> m}\langle v\rangle^{m}K_l\chi_{\leq m} g(s)\|_{L^1} ds\big]\nonumber\\
\lesssim &\Phi\delta_{m,0}\int_{0}^{t}\|\chi_{> m}\langle v\rangle^{m+1}g(s)\|_{L^1}\ ds+\Phi \Upsilon_m\int_{0}^{t}\|\chi_{\leq m}\langle v\rangle^{m+1}g(s)\|_{L^1}\ ds\nonumber\\
\lesssim &\big[ \delta_{m,0}+\frac{\Upsilon_m}{\Phi}\big]\|g\|_{t}. \label{eq:contr3}
\end{align}
where $\delta_{m,0}$ is defined before Lemma \ref{LM:Control}. It shows that $\delta_{m,0}\rightarrow 0$ as $m\rightarrow \infty$.

The estimates in \eqref{eq:contr1}, \eqref{eq:contr2} and \eqref{eq:contr3} imply that, for some constant $c$,
\begin{align}
\|H(g)\|_{t}\leq 
&c\big[\frac{\Upsilon_m}{\Phi} +\delta_{m,0}+(\Phi+1)\epsilon(tm) \Upsilon_m\big]\|g\|_{t}.\label{eq:contractive}
\end{align}

Now we choose $m$, $\Phi$ and $t$ to make
\begin{align}
c\big[\frac{\Upsilon_m}{\Phi} +\delta_{m,0}+(\Phi+1)\epsilon(tm) \Upsilon_m\big]\leq \frac{3}{4}.\label{eq:thrfor}
\end{align}  This, together with \eqref{eq:contractive}, implies the desired \eqref{eq:fix2}.

It is easy to choose $m$, $\Phi$ and $t$ to make \eqref{eq:thrfor} hold. Specifically,
first choose $m$ to be sufficiently large, so that $$c \delta_{m,0}\leq \frac{1}{4},$$ secondly choose $\Phi$ so that $$c\frac{\Upsilon_m}{\Phi}\leq \frac{1}{4},$$ and lastly choose $t$ small enough so that
$$c(\Phi+1)\epsilon(tm) \Upsilon_m\leq \frac{1}{4}.$$

The proof of \eqref{eq:fix2} is complete.

The proof of \eqref{eq:fix1} is considerably easier, hence omitted.

\end{proof}

\def\cprime{$'$} \def\cprime{$'$} \def\cprime{$'$} \def\cprime{$'$}


\end{document}